\documentclass[12pt,reqno]{article}

\usepackage[margin=1in]{geometry}
\usepackage[all]{xy}
\usepackage[utf8]{inputenc}
\usepackage{pdfpages}
\usepackage{amsmath}
\usepackage{amssymb}
\usepackage{csquotes}
\usepackage{amsthm}
\usepackage{enumerate}
\usepackage{textcomp}
\usepackage{todonotes}
\usepackage[titletoc,toc,title]{appendix}
\usepackage[colorlinks,citecolor={blue},linkcolor={blue},linktocpage=true]{hyperref}
\usepackage{makeidx}
\usepackage[refpage]{nomencl}[0]
\usepackage{mathtools}
\usepackage[makeroom]{cancel}
\usepackage{graphicx,nicefrac}
\usepackage{authblk}
\usepackage[symbol*]{footmisc}
\usepackage[makeroom]{cancel}
\usepackage{indentfirst}
\usepackage{fancyhdr}
\usepackage{cleveref}

\makeindex

\newcommand{\Addresses}{{
		\bigskip
		\footnotesize
		
		\textsc{School of Mathematical Sciences, Tel Aviv University, Ramat Aviv, Tel Aviv 6997801, Israel}\par\nopagebreak
		\textit{E-mail address:} \href{mailto:zahihaza@tauex.tau.ac.il}{\color{black}{\textbf{zahihaza@tauex.tau.ac.il}}}
		
}}

\makenomenclature

\numberwithin{equation}{section}

\newcommand{\bbc}{\mathbb{C}}

\newcommand{\bba}{\mathbb{A}}

\newcommand{\bbr}{\mathbb{R}}
\newcommand{\bbz}{\mathbb{Z}}
\newcommand{\ocal}{\mathcal{O}}

\newcommand{\pcal}{\mathcal{P}}
\newcommand{\pomega}{\varpi}

\newcommand{\gl}[2][\bba]{\mathrm{GL}_{#2}(#1)}

\newcommand\restr[2]{{
		\left.\kern-\nulldelimiterspace 
		#1 
		\vphantom{\big|} 
		\right|_{#2} 
}}

\theoremstyle{plain}

\newtheorem{Theorem}{Theorem}
\newtheorem*{Theorem*}{Theorem}

\newtheorem{seclem}{Lemma}[section]
\newtheorem{claim}[seclem]{Claim}
\newtheorem{cor}[seclem]{Corollary}
\newtheorem{prop}[seclem]{Proposition}
\newtheorem*{prop*}{Proposition}

\theoremstyle{definition}

\theoremstyle{remark}
\newtheorem{Rem*}{Remark}
\newtheorem{Rem}[seclem]{Remark}

\DefineFNsymbolsTM{otherfnsymbols}{%
  1 1
  2 2
}%

\setfnsymbol{otherfnsymbols}

\author{Zahi Hazan}

\title{An Identity relating Eisenstein series on general linear groups}
\date{\vspace{-5ex}}

\makeatletter
\newcommand{\subjclass}[2][1991]{%
  \let\@oldtitle\@title%
  \gdef\@title{\@oldtitle\footnotetext{#1 \emph{Mathematics subject classification.} #2}}%
}
\newcommand{\keywords}[1]{%
  \let\@@oldtitle\@title%
  \gdef\@title{\@@oldtitle\footnotetext{\emph{Key words and phrases.} #1.}}%
}
\newcommand{\proofpart}[2]{%
	\par
	\addvspace{\medskipamount}%
	\noindent\emph{Part #1: #2}\par\nobreak
	\addvspace{\smallskipamount}%
	\@afterheading
}
\makeatother


\begin{document}

\maketitle

\abstract{
We give a general identity relating Eisenstein series on general linear groups. We do it by constructing an Eisenstein series, attached to a maximal parabolic subgroup and a pair of representations, one cuspidal and the other a character, and express it in terms of a degenerate Eisenstein series. In the local fields analogue, we prove the convergence in a half plane of the local integrals, and their meromorphic continuation. In addition, we find that the unramified calculation gives the Godement-Jacquet zeta function. This realizes and generalizes the construction proposed by Ginzburg and Soudry in Section 3 in \cite{ginzburgintegrals}.
}

\section{Introduction}

\subsection{Statement of the main results}\label{statement_sec}
\subsubsection*{Global theory}
Let $k$ be a number field, and let $\bba$ be its ring of adeles.
Let $m\geq n$ be two positive integers. Let $\pi$ be an irreducible, automorphic, cuspidal representation of $\gl{n}$ and assume that its central character $\omega_\pi$ is unitary. Let $\varphi_{\pi}$ be a cusp form in the space of $\pi$. Consider, for a complex number $s$, the Eisenstein series $\mathrm{E}_{m n}\left(f_{\omega_\pi,s}\right)$ on $\gl{m n}$, attached to a smooth, holomorphic section $f_{\omega_\pi,s}$ of the normalized parabolic induction  $\mathrm{Ind}_{P_{m n-1,1}(\bba)}^{\gl{m n}}\left(1\otimes \omega_\pi^{-1}\right) \delta_{P_{m n-1,1}}^{s}$.
We set
\begin{equation}\label{eisen_def2}
\mathfrak{E}\left(f_{\omega_\pi,s},\varphi_{\pi}\right)(h)=\int \limits _{Z_n(\bba)\gl[k]{n}\backslash\gl{n}}\varphi_{\pi}(g)\mathrm{E}_{m n}\left(f_{\omega_\pi,s}\right)\left(t\left(h,g\right)\right)dg.
\end{equation}
Here, $Z_n(\bba)$ denotes the center of $\gl{n}$ and the map $t:\gl{m}\times\gl{n}\rightarrow\gl{mn}$ denotes the Kronecker product (see \Cref{kprod_sec}).
Our main theorem is the identity
\begin{Theorem}\label{main_thrm}
	$\mathfrak{E}\left(f_{\omega_\pi,s},\varphi_{\pi}\right)$ is an Eisenstein series on $\gl{m}$, corresponding to the normalized parabolic induction	$\mathrm{Ind}_{P_{m-n,n}(\bba)}^{\gl{m}}\left(1\otimes \tilde{\pi}\right) \delta_{P_{m-n,n}}^{s}.$
	In more details, there exists an explicit meromorphic section $\xi\left(f_{\omega_\pi,s},\varphi_{\pi}\right)$ of this parabolic induction, given by the following integral which converges absolutely for $\Re(s)$ sufficiently large and admits a meromorphic continuation to $\bbc$,
	\begin{equation}\label{global_section_f_intro}
	\xi\left(f_{\omega_\pi,s},\varphi_{\pi}\right)( h) = \int \limits _{Z_n(\bba)\backslash\gl{n}}\varphi_{\pi}(g)  f_{\omega_\pi,s}\left(\tilde{\varepsilon}   t\left(h,g\right)\right)dg,
	\end{equation}
	where $\tilde{\varepsilon}$ is a specific element of $\gl[k]{mn}$. For $\Re(s)$ sufficiently large,
	\begin{equation}\label{eisen_ident2_exp}
	\mathfrak{E}\left(f_{\omega_\pi,s},\varphi_{\pi}\right)(h)=\sum\limits_{\gamma\in P_{m-n,n}(k)\backslash\gl[k]{m}}\xi\left(f_{\omega_\pi,s},\varphi_{\pi}\right)(\gamma h).
	\end{equation}
	The R.H.S of \cref{eisen_ident2_exp} continues to a meromorphic function in the whole plane. Denote it by $\mathrm{E}_{m}\left(\xi\left(f_{\omega_\pi,s},\varphi_{\pi}\right)\right)$. Then, as meromorphic functions on
	$\gl{m}$,
	\begin{equation}\label{eisen_ident2}
	\mathfrak{E}\left(f_{\omega_\pi,s},\varphi_{\pi}\right) = \mathrm{E}_{m}\left(\xi\left(f_{\omega_\pi,s},\varphi_{\pi}\right)\right).
	\end{equation}	
\end{Theorem}
\begin{Rem}\label{realization_rem}
	The identity \eqref{global_section_f_intro} immediately gives a realizations of $\pi$, in the case $m=n$.
\end{Rem}

The proof of this theorem occupies the next two sections. In \Cref{global_result_section} we formally show that the global identity \cref{eisen_ident2_exp} holds. In \Cref{local_sec} we make sense of the integral in \cref{global_section_f_intro}. This allows us to complete the proof of \Cref{main_thrm}.

\subsubsection*{Local theory}
Let $\nu$ be a finite place of $k$. Denote by $k_{\nu}$ the completion of $k$ with respect to $\nu$. Denote by $\ocal_{\nu}$ its ring of integers. Let $\pi_{\nu}$ be a smooth, irreducible, generic representation of $\gl[k_{\nu}]{n}$ (with central character $\omega_{\pi_\nu}$).
Our global integral gives rise to the local integral at $\nu$:
\begin{equation}
I\left(f_{\omega_{\pi_\nu},s},v_{\pi_{\nu}}\right)=\int \limits _{Z_n\left({k_{\nu}}\right)\backslash \gl[k_{\nu}]{n}}f_{\omega_{\pi_\nu},s}\left(\tilde{\varepsilon}t\left(I_m, g\right)\right)\pi_{\nu}\left(g_\nu\right)v_{\pi_{\nu}}dg_\nu,
\end{equation}
where $v_{\pi_{\nu}}$ is in the space of $\pi_{\nu}$, and $f_{\omega_{\pi_\nu},s}$ is a smooth, holomorphic section of the induced representation $\mathrm{Ind}_{P_{mn-1,1}(k_\nu)}^{\gl[k_\nu]{mn}}\left(1\otimes \omega_{\pi_\nu}^{-1}\right) \delta_{P_{mn-1,1}}^{s}$.

In \Cref{bad_places_sec} we prove
\begin{Theorem}\label{local_int_is_pol_prop}
	Assume $\Re(s)>>0$. If $\nu$ is an Archimedean place, then $I\left(f_{\omega_{\pi_\nu},s},v_{\pi_{\nu}}\right)(h_\nu)$ is a meromorphic function. If $\nu$ is a ramified non-Archimedean place, then
	\begin{equation}\label{local_int_is_rational_id}
	I\left(f_{\omega_{\pi_\nu},s},v_{\pi_{\nu}}\right)(h_\nu)=P_\nu\left(q_\nu^{-s}\right),
	\end{equation}
	where $P_\nu\left(q_\nu^{-s}\right)$ is a rational function in $\bbc\left(q_\nu^{-s}\right)$.
	Moreover, $I\left(f_{\omega_{\pi_\nu},s},v_{\pi_{\nu}}\right)(h_\nu)$ extends to a meromorphic function on all $\mathbb{C}$, for all place $\nu$,.
\end{Theorem}

Assume that $\pi_{\nu}$ is unramified (this is the case for all but finitely many places $\nu$). Fix a spherical vector $\check{v}_{\pi_{\nu}}^\circ$ in $V_{\pi_\nu}^*$. Then, there exists a unique spherical vector $v_{\pi_{\nu}}^\circ$ in $V_{\pi_\nu}$, such that $\left<v^\circ_{\pi_{\nu}},\check{v}_{\pi_{\nu}}^\circ\right>=1$. Similarly, there exists a unique unramified section $f_{\omega_{\pi_\nu},s}^\circ$ normalized by the condition $f_{\omega_{\pi_\nu},s}^\circ(I_{mn})=1$.

Given a matrix coefficient $c_{v,\check{w}}(g)=\left<\pi(g)v,\check{w}\right>$ of $\pi_{\nu}$, where $v\in V_{\pi_\nu}$ and $\check{w}\in V_{\tilde{\pi}_\nu}$, and given a Schwartz-Bruhat function $\Phi\in S\left(M_n(k_{\nu})\right)$, we recall the Godement-Jacquet zeta integral \cite{godement1972local}
\begin{equation}\label{GJ_zeta}
Z_{\mathrm{GJ}}\left(s,c_{v,\check{w}},\Phi\right)=\int\limits_{\gl[k_\nu]{n}}\Phi(g_\nu)c_{v,\check{w}}(g_\nu)\left|\det g_\nu\right|^{s+\frac{n-1}{2}}dg_\nu,
\end{equation} 
which is absolutely convergent for $\Re (s)$ sufficiently large. In \Cref{unram_sec} we prove
\begin{Theorem}\label{unrm_calc_GJ_thrm}
	\begin{equation}\label{unrm_calc_GJ_id}
		I\left(f_{\omega_{\pi_\nu},s}^\circ,v_{\pi_{\nu}}^\circ\right)= \frac{Z_{\mathrm{GJ}}\left(m(s+\frac{1}{2})-\frac{n-1}{2},c_{v_{\pi_{\nu}}^\circ,\check{v}_{\pi_{\nu}}^{\circ}},\Phi_0\right)}{L(m(s+\frac{1}{2}),\omega_{\pi_\nu})}v_{\pi_{\nu}}^\circ,
	\end{equation}
	where $\Phi_0$ is the characteristic function of $M_n\left(\ocal_{\nu}\right)$.
\end{Theorem}
By the test vector lemma \cite[Lemma 6.10]{godement1972local} we have $Z_{\mathrm{GJ}}(s,c_{v_{\pi_{\nu}}^\circ,\check{v}_{\pi_{\nu}}^{\circ}},\Phi_0)= L(s,\pi_\nu)$. Thus, \Cref{unrm_calc_GJ_thrm} immediately implies
\begin{Theorem}\label{unrm_calc_L_thrm}
	\begin{equation}
		I\left(f_{\omega_{\pi_\nu},s}^\circ,v_{\pi_{\nu}}^\circ\right)= \frac{L\left(m(s+\frac{1}{2})-\frac{n-1}{2},\pi_\nu\right)}{L\left(m(s+\frac{1}{2}),\omega_{\pi_\nu}\right)}v_{\pi_{\nu}}^\circ.
	\end{equation}
\end{Theorem}

\subsection{Background and motivation}

Eisenstein series are key objects in the theory of automorphic forms. They are an important tool in the study of automorphic $L$-functions, and they figure out in the spectral decomposition of the $L^2$-space of automorphic forms. In recent years, new constructions of global integrals generating identities relating Eisenstein series were discovered.

In \cite{ginzburg2018two} Ginzburg and Soudry introduced two general identities relating Eisenstein series on split classical groups, as well as double covers of symplectic groups.
Basically, the idea of their theorem, for example, in case of symplectic group, is as follows.
Let $\tau$ be an irreducible, automorphic, cuspidal representation of $\gl{n}$. Denote by $\Delta(\tau,i)$ the Speh representation of $\mathrm{GL}_{ni}(\bba)$ of \enquote{length} $i$ corresponding to $\tau$. Let $\sigma$ be an irreducible, automorphic, cuspidal representation of $\mathrm{Sp}_m(\bba)$, and $\sigma^\iota$ a certain outer conjugation of $\sigma$ by an element of order 2.
Their theorem shows that Eisenstein series parabolically induced from $\Delta(\tau,i)\left|\det\cdot\right|^s\times \sigma ^\iota$ can be expressed in terms of \enquote{more degenerate} Eisenstein series, namely Eisenstein series parabolically induced from the Speh representation $\Delta(\tau,m+i)\left|\det\right|^s$.
In particular they get that any Eisenstein series attached to an irreducible, cuspidal representation on a maximal parabolic subgroup can be expressed in terms of an Eisenstein series, attached to a Speh representation on a Siegel parabolic subgroup.
They generalized M{\oe}glin's work \cite{moeglin1997quelques} which is an extension of \cite{ginzburg1997functions}. The identity of Ginzburg and Soudry can also be viewed as an extension of the doubling construction introduced in \cite{cai2016doubling}.
The second identity in \cite{ginzburg2018two} generalizes Ikeda's work \cite{ikeda1994} and can be viewed as a generalization of the descent construction studied in \cite{ginzburg2011descent}.

In this work we demonstrate the principle above by proving an identity relating Eisenstein series on general linear groups. Namely, we express an Eisenstein series, attached to a maximal parabolic subgroup and a pair of representations, one cuspidal and the other a character, in terms of a degenerate Eisenstein series.

The identity \eqref{eisen_ident2} has further applications. In an ongoing project, we explore several of these. For example, we use it to provide another proof that the Eisenstein series $\mathrm{E}_{m}\left(\xi\left(f_{\omega_\pi,s},\varphi_{\pi}\right)\right)$ is holomorphic, and also apply the aforementioned realization (see \Cref{realization_rem})  to certain Rankin-Selberg integrals.

\subsection{Preliminaries and notation}
\subsubsection{The groups}
Recall that $k$ denotes a number field, and $\bba$ its ring of adeles. We consider general linear groups and their parabolic subgroups, as algebraic groups over $k$. Let $\ell,r\geq 0$ be two integers. We write $P_{\ell,r}$ for the block upper-triangular maximal parabolic subgroup of $\mathrm{GL}_{\ell+r}$ with Levi part $\mathrm{S}_{\ell,r}\cong\mathrm{GL}_{\ell}\times\mathrm{GL}_{r}$. 
Its Levi decomposition is
\begin{align}
&{P_{\ell,r}= \mathrm{S}_{\ell,r}\ltimes U_{\ell,r}},\\
&\mathrm{S}_{\ell,r} = \left\{\begin{pmatrix}
A&0\\0&B
\end{pmatrix}|\ A\in\mathrm{GL}_{\ell},\ B\in\mathrm{GL}_{r} \right\},&\label{levi_of_parab}\\
&U_{\ell,r} = \left\{\begin{pmatrix}
I_\ell&X\\0&I_r
\end{pmatrix}|\ X\in M_{\ell,r}\right\}.&
\end{align}

We denote by $\mathrm{H}_{\ell,1}$ (or $\mathrm{H}_{1,r}$) the subgroup of $P_{\ell,1}$ (or $P_{1,r}$) such that  $B=1$ (or $A=1$) in \cref{levi_of_parab}.
The standard Borel subgroup of $\mathrm{GL}_{\ell}$ is denoted by $P_\ell=S_\ell U_{\ell}$, and the Weyl group of $\mathrm{GL}_{\ell}$ by $W_{\ell}$.

For a place $\nu$, we let $k_{\nu}$ be the completion of $k$ with respect to the absolute value $|\cdot|_\nu$. Let $\ell\geq 1$ be an integer. We denote the maximal compact subgroup of $\gl[k_\nu]{\ell}$ by $K_{\ell,\nu}$. For $\nu = \mathbb{R}$, we have $K_{\ell,\nu}=\mathrm{O}_\ell$, the orthogonal subgroup, and for $\nu=\mathbb{C}$, we have $K_{\ell,\nu}=\mathrm{U}_\ell(\mathbb{C})$, the unitary subgroup (not to be confused with the unipotent radical). For $\nu < \infty$, we denote the ring of integers of $k_\nu$ by $\ocal_{\nu}$ and its maximal ideal by $\pcal_\nu$. We denote a uniformizer of $\pcal_\nu$ by $\pomega$, and the cardinality of the residue field by $q_\nu$. In this case, $K_{\ell,\nu}=\gl[\ocal_{\nu}]{\ell}$. 
We denote $K_\ell:=\prod_{\nu}K_{\ell,\nu}$.

Let $P=SU$ be a standard parabolic subgroup of $G=\mathrm{GL}_\ell$. Let $T_{S}=Z\left(S\right)$ be the center of $S$. For simplicity we denote $T=T_{S_\ell}$. Let $\Sigma$ be the set of all roots corresponding to the pair $\left(G,T\right)$. i.e., the non-trivial eigencharacters of the adjoint action of $T$ on the lie algebra $\mathfrak{g}$ of $G$. Let $\Sigma^+$ be the subset of positive roots determined by $P_\ell$, so that $\mathfrak{u}=\bigoplus_{\alpha\in \Sigma ^+}\mathfrak{g}_\alpha$, where $\mathfrak{g}_\alpha$ is the eigenspace of $\alpha$, and $\mathfrak{u}$ is the lie algebra of $U_\ell$. Let $\Delta=\Delta_{P_\ell}$ be the basis of $\Sigma^+$, so that every root in $\Sigma^+$ is a sum of roots in $\Delta$.

\subsubsection{Sections and Eisenstein series}\label{parab_ind_sec}

We recall the definition of a holomorphic section of a parabolically induced representation, parameterized by an unramified character of the Levi part (see \cite[\S 2.4]{kaplan2013local}). For a thorough treatment of this subject refer to \cite[\S IV]{waldspurger2003formule} and \cite{muic2008geometric}.
Let $\ell$ and $r$ be two positive numbers. Let $\tau$ an automorphic representation of $\gl{\ell}$. Consider, for a complex number $s$, the normalized parabolic induction
\begin{equation*}
	\rho_{\tau,s} = \mathrm{Ind}_{P_{\ell,r}(\bba)}^{\gl{\ell+r}}\left(1\otimes \tau\right) \delta_{P_{\ell,r}}^{s}.
\end{equation*}
Here, $\delta_{P_{\ell,r}}$ stands for the modulus character of $P_{\ell,r}(\bba)$, i.e.,
\begin{equation*}
	\delta_{P_{\ell,r}}\left(\begin{pmatrix}A&X\\&B\end{pmatrix}\right)=\left|\det A\right|^r\left|\det B\right|^{-\ell},
\end{equation*}
where $A\in \gl{\ell}$, and $B\in \gl{r}$, and $X\in M_{\ell,r}(\bba)$.

For a given complex number $s$, the representation $\rho_{\tau,s}$ acts in the space of all smooth, holomorphic functions $\tilde{f}_{\tau,s}:\gl{\ell+r}\times \gl{\ell}\times \gl{r}\to\bbc$ that satisfy
\begin{itemize}
	\item [I.] For all $g\in\gl{\ell+r},\ a,A\in\gl{\ell},\ b,B\in\gl{r},$ and $X\in \mathrm{M}_{\ell,r}(\bba)$,
	\begin{equation*}
	\tilde{f}_{\tau,s}\left(\begin{pmatrix}A&X\\&B\end{pmatrix}g;a,b\right)=
	\delta_{P_{\ell,r}}^{s+\nicefrac{1}{2}}\left(\begin{pmatrix}A&\\&B\end{pmatrix}\right)\tilde{f}_{\tau,s}\left(g;aA,bB\right).
	\end{equation*} 
	\item [II.] For fixed $g\in\gl{\ell+r}$,
	\begin{equation*}
		\left[\left(a,b\right)\mapsto \tilde{f}_{\tau,s}\left(g;a,b\right)\right]\in \tau,
	\end{equation*} 
	where $a\in\gl{\ell}$ and $b\in\gl{r}$.
\end{itemize}
By smooth we mean that there is some compact open subgroup $Y\leq \gl{\ell+r}$ such that $\tilde{f}_{\tau,s}(ng)=\tilde{f}_{\tau,s}(g)$ for all $n\in Y$ and $g\in \gl{\ell+r}$.
We realize the space of $\rho_{\tau,s}$ as smooth, holomorphic functions from $\gl{\ell+r}$ to $\mathbb{C}$ by setting
\begin{equation*}
	f_{\tau,s}\left(g\right)=\tilde{f}_{\tau,s}\left(g;I_\ell,I_r\right).
\end{equation*}
The function $f_{\tau,s}$ satisfies
\begin{equation*}
f_{\tau,s}\left(\begin{pmatrix}A&X\\&B\end{pmatrix}g\right)=\delta_{P_{\ell,r}}^{s+\nicefrac{1}{2}}\left(\begin{pmatrix}A&X\\&B\end{pmatrix}\right)\tilde{f}_{\tau,s}\left(g;A,B\right).
\end{equation*}
In particular,
\begin{equation*}
	f_{\tau,s}:\mathrm{S}_{\ell,r}(k) U_{\ell,r}(\bba)\backslash\gl{\ell+r}\to\mathbb{C}.
\end{equation*}

There is a bijection between $\restr{\rho_{\tau,s}}{K_{\ell+r}}$ and $\mathrm{Ind}_{P_{\ell,r}(\bba)\cap K_{\ell+r}}^{K_{\ell+r}}\restr{\left(1\otimes \tau\right)}{P_{\ell,r}(\bba)\cap K_{\ell+r}}$. Let $g=g_Pg_K$ were $g_P\in P_{\ell,r}(\bba)$ and $g_K\in K_{\ell+r}$ be the Iwasawa decomposition of $g\in \gl{\ell+r}$. This bijection is given by mapping $\varphi_\tau\in \mathrm{Ind}_{P_{\ell,r}(\bba)\cap K_{\ell+r}}^{K_{\ell+r}}\restr{\left(1\otimes \tau\right)}{P_{\ell,r}(\bba)\cap K_{\ell+r}}$ to $f_{\varphi_\tau,s}$:
\begin{equation*}
	f_{\varphi,s}(g)=\delta_{P_{\ell,r}}(g_P)\left(1\otimes \tau\right)(g_P)\varphi_\tau(g_K).
\end{equation*}
A section of the form $f_{\varphi_\tau,s}(g)$ is called a standard section (when restricted to the maximal compact subgroup, it does not depend on $s$). The space of holomorphic sections equals to the space of all linear combinations of standard sections over $\bbc[q_\nu^{-s},q_\nu^s]$. Hence, a holomorphic section $f_{\tau,s}$ can be written as
\begin{equation*}
	f_{\tau,s}=\sum_{i=1}^{N}P_i(q_\nu^{-s}, q_\nu^s)f_{\varphi^{(i)}_\tau,s},
\end{equation*}
where for all $1\leq i \leq N$, $P_i \in\bbc[q_\nu^{-s}, q_\nu^s]$ and $f_{\varphi^{(i)}_\tau,s}$ is a standard section.

We use similar notations over $k_\nu$. Let $\tau_\nu$ be a representation of $\gl[k_\nu]{\ell}$. Consider, for a complex number $s$, the normalized parabolic induction
\begin{equation*}
	\rho_{\tau_\nu,s} = \mathrm{Ind}_{P_{\ell,r}(k_\nu)}^{\gl[k_\nu]{\ell+r}}\left(\tau_\nu\right) \delta_{P_{\ell,r}}^{s},
\end{equation*}
where $\delta_{P_{\ell,r}}$ is the modulus character of $P_{\ell,r}(k_\nu)$. We use analogous notation realize the space of $\rho_{\tau_\nu,s}$ as functions $f_{\tau_\nu,s}$ from $U_{\ell,r}(k_\nu)\backslash\gl[k_\nu]{\ell+r}$ to $V_{\tau_\nu}$. Similarly we can write a holomorphic section $f_{\tau_\nu,s}$ as
\begin{equation}\label{hol_sec_as_standard_local}
	f_{\tau_\nu,s}=\sum_{i=1}^{N}P_i(q_\nu^{-s}, q_\nu^s)f_{\varphi_{i,\nu},s},
\end{equation}
where for all $1\leq i \leq N$, $P_i$ is holomorphic function in $q_\nu^{\pm s}$ (if $\nu$ is non-Archimedean then $P_i\in\bbc[q_\nu^{-s}, q_\nu^s]$) and $f_{\varphi_{i,\nu},s}$ is a standard section in $\rho_{\tau_\nu,s}$.

We consider smooth, holomorphic sections $s\mapsto f_{\tau,s}$ of the normalized parabolic induction $\rho_{\tau ,s}$. We denote by $\mathrm{E}_{\ell+r}\left(f_{\tau,s}\right)$ the Eisenstein series on $\gl{\ell+r}$, attached to $f_{\tau,s}$. For $\Re(s)$ sufficiently large, it is given by the following (absolutely convergent) series
\begin{equation*}
\mathrm{E}_{\ell+r}\left(f_{\tau,s}\right)(h)=\sum \limits _{\gamma\in P_{\ell,r}(k)\backslash \gl[k]{\ell+r}} f_{\tau,s} \left(\gamma h\right).
\end{equation*}

\subsubsection{Kronecker product}\label{kprod_sec}

Let $F$ be a field. We realize the tensor product map $t_F:\gl[F]{\ell}\times \gl[F]{r}\to \gl[F]{\ell r}$ as follows.
	Let $h$ and $g$ be two square matrices of sizes $\ell$ and $r$, respectively. Then, $t_F(h,g)$ is the $\ell r$ square block matrix
	\begin{equation}\label{kprod_eq}
		t_F(h,g)= \begin{pmatrix}
			h_{1,1}g&\cdots&h_{1,\ell}g\\\vdots&&\vdots\\h_{\ell,1}g&\cdots&h_{\ell,\ell}g
		\end{pmatrix},
	\end{equation}
	where $h=(h_{i,j})_{1\leq i,j\leq \ell}$.

By \cref{kprod_eq} we immediately get that
$\ker t_F = \{(\lambda I_\ell,\lambda^{-1} I_r)|\ \lambda\in F^\times\}$. Therefore,
\begin{equation}\label{kron_grp}
	T_{\ell,r}(F):=\mathrm{Im} t_F\cong F^\times\backslash (\mathrm{GL}_\ell(F) \times \mathrm{GL}_r(F)). 
\end{equation}
It is convenient to simply denote $t_F=t$. 

We denote the transpose of the matrix $X$ by $X^T$. For a square matrix $Y$, we denote its determinant by $|Y|$. If $Y$ is also invertible we set $Y^{*}:=\left(Y^T\right)^{-1}$. The Kronecker product satisfies
\begin{eqnarray}\label{kprod_prop}
&t(h,g)=t(h,I_r)t(I_\ell,g)=t(I_\ell,g)t(h,I_r),\\
&	(t(h,g) )^{T}= t(h^T,g^T),\ 
	\left| t(h,g) \right|=\left| {h} \right|^{r}\left| {g} \right|^{\ell},\ \mathrm{and} \ ( t(h,g))^{*}= t(h^*,g^*).
\end{eqnarray}

\section{Proof of \Cref{main_thrm} - global theory} \label{global_result_section}

In this section we prove the global unfolding part of \Cref{main_thrm}. First, we note that the integral of \cref{eisen_def2} is absolutely convergent. This is due to the rapid decrease of the cusp form, the moderate growth of the Eisenstein series and the fact that the domain of integration is of finite measure.

We start with unfolding the Eisenstein series $E_{mn}\left(f_{{\omega_\pi},s}\right)$ for $\Re(s)>>0$ in \cref{eisen_def2}, where it is defined by
	\begin{equation}\label{eisen_def}
	E_{mn}\left(f_{\omega_\pi,s}\right)(t\left(h,g\right)) = \sum\limits _{\varepsilon\in P_{mn-1,1}(k)\backslash \gl[k]{mn}} f_{\omega_\pi,s}\left(\varepsilon  t\left(h,g\right)\right).
	\end{equation}
    The group $T_{m,n}(k)$ acts on the set of cosets $P_{mn-1,1}(k)\backslash \mathrm{GL}_{mn}(k)$ from the right. We split the sum as
	\begin{equation}\label{unfold_g6}
		E_{mn}\left(f_{\omega_\pi,s}\right)(t\left(h,g\right)) =\sum\limits _{\varepsilon\in P_{mn-1,1}(k)\backslash \gl[k]{mn}/T_{m,n}(k)} \sum \limits _{\gamma \in Q^\varepsilon(k)\backslash T_{m,n}(k)}  f_{\omega_\pi,s}\left(\varepsilon \gamma t\left(h,g\right)\right),	
	\end{equation}
	where $Q^\varepsilon := P_{mn-1,1}^\varepsilon \cap T_{m,n}$ and $ P_{mn-1,1}^\varepsilon := \varepsilon^{-1} P_{mn-1,1}\varepsilon$.

	In \Cref{double_coset_sec} we show that the set $P_{mn-1,1}(k)\backslash \mathrm{GL}_{mn}(k)/T_{m,n}(k)$ is finite and  find an explicit set of representatives.
	In \Cref{stab_sec} we find for each representative $\varepsilon$, the stabilizer $Q^\varepsilon$.
	In \Cref{Only_one_cont_sec} we show that the representative corresponding to the open cell is the only one that contributes to the integral \cref{eisen_def2}. Then, we rewrite the integral in \cref{eisen_def2} in the form of an Eisenstein series on $\gl{mn}$ attached to a section $\xi\left(f_{\omega_\pi,s},\varphi_{\pi}\right)$ which we write explicitly.
		
\subsection{The double cosets $P_{mn-1,1}(k)\backslash \mathrm{GL}_{mn}(k)/T_{m,n}(k)$}\label{double_coset_sec}

As mentioned above, in this section we show that the set  $P_{mn-1,1}(k)\backslash \mathrm{GL}_{mn}(k)/T_{m,n}(k)$ is finite and find an explicit set of representatives. Namely, we prove
\begin{Theorem}\label{double_cosets_rep_thrm}
	There are exactly $n$ double cosets in $P_{mn-1,1}(k)\backslash \mathrm{GL}_{mn}(k)/T_{m,n}(k)$. We list the following set of representatives: for $0\leq r \leq n-1$
	\begin{equation}\label{doubel_coset_reps_form}
	\varepsilon_r:=\begin{pmatrix}
	I_{(m-r)n-1}&&\\&&I_{rn}\\&1&\end{pmatrix}\begin{pmatrix}
	I_{(m-r)n-1}&0&0\\&1&\underline{b_r}\\ &&I_{rn}\end{pmatrix},
	\end{equation} 
	where $\underline{b_r}:=\left(e_{n-1}^T,e_{n-2}^T,\ldots,e_{n-r}^T\right)$.
\end{Theorem}

The rest of this section is devoted to the proof of \Cref{double_cosets_rep_thrm}.

We begin as follows. Using the Bruhat decomposition in $\mathrm{GL}_{mn}(k)$, we have
		\begin{equation}\label{first_split_to_double_coset}
			\mathrm{GL}_{mn}(k)=\bigcup \limits_{1\leq j \leq mn}\bigcup \limits_{u\in U^{(j)}_{mn}(k)}P_{mn-1,1}(k)w_j	u T_{m,n}(k),
		\end{equation}
		where $U^{(j)}_{mn}(k):=w_j^{-1}P_{mn-1,1}(k)w_j\cap U_{mn}(k)\backslash U_{mn}(k)$, and for $ 1\leq j\leq mn$, $w_j$ are the representatives of $\begin{pmatrix} W_{mn-1}&\\&1\end{pmatrix}\backslash W_{mn}$, i.e.
		\begin{equation}\label{weyl_representatives}
			w_j=\begin{pmatrix}
			I_{j-1}&&\\&&I_{mn-j}\\&1&\end{pmatrix}.
		\end{equation}
The following lemma gives an explicit form of the representatives in \cref{first_split_to_double_coset}.
		\begin{seclem}\label{uni_reps_claim}
			Let $1\leq j\leq mn$. The elements
			\begin{equation*}
			u_{j}(\underline{v_j}):=\begin{pmatrix}
			I_{j-1}&0&0\\&1&\underline{v_j}\\ &&I_{mn-j}\end{pmatrix},\qquad \underline{v_j}\in k^{mn-j}.
			\end{equation*}
			form a set of representatives of $U^{(j)}_{mn}(k)$.
		\end{seclem}
	
		\begin{proof}
		Let
		\begin{equation*}
			u=\begin{pmatrix}
			u_{j-1}&X_1&X_2\\&1&\underline{v_j}\\ &&u_{mn-j}\end{pmatrix}\in U_{mn}(k).
		\end{equation*}
		Then,
		\begin{equation*}
		w_j	\begin{pmatrix}
		u_{j-1}&X_1&X_2\\&1&\underline{v_j}\\ &&u_{mn-j}\end{pmatrix}w_j^{-1}=\begin{pmatrix}
			u_{j-1}&X_2&X_1\\&u_{mn-j}&0\\ &\underline{v_j}&1\end{pmatrix}.
		\end{equation*}
		Thus, $	w_j	u w_j^{-1}\in P_{mn-1,1}(k)$ iff $\underline{v_j}=\underline{0}$. This gives 
		\begin{equation*}
			U_{mn}(k)\cap w_j^{-1} P_{mn-1,1}(k) w_j=\left\{\begin{pmatrix}
			u_{j-1}&X_1&X_2\\&1&0\\ &&u_{mn-j}\end{pmatrix}\in U_{mn}(k)\right\}
		\end{equation*}
		and the lemma follows.
		\end{proof}
		
		We denote
		\begin{equation*}
			C_j(\underline{v_j}):=P_{mn-1,1}(k)w_j	u_{j}(\underline{v_j}) T_{m,n}(k).
		\end{equation*}
		In this notation we can rewrite the decomposition in \cref{first_split_to_double_coset} as follows.
		\begin{equation}\label{second_split_to_double_coset}
			\mathrm{GL}_{mn}(k)=\bigcup \limits_{1\leq j \leq mn}\bigcup \limits_{\underline{v_j}\in k^{mn-j}}C_j(\underline{v_j}).
		\end{equation}

     We now note that in particular, for $r=0$ we have $\varepsilon_0=I_{mn}$ and $C_{mn}(\underline{b_0})=P_{mn-1,1}(k)T_{m,n}(k)$. For $0\leq r\leq n-1$ the representative $\varepsilon_r$ corresponds to the double coset $C_{(m-r)n}(\underline{b_r})$. Therefore, \Cref{double_cosets_rep_thrm} reduces the decomposition in \cref{second_split_to_double_coset} to the disjoint union
     \begin{equation}\label{split_to_double_coset}
     \mathrm{GL}_{mn}(k)=  \bigcup\limits_{0\leq r \leq n-1}C_{(m-r)n}(\underline{b_r}).
     \end{equation}

  	We continue by viewing the first row of $u_{1}(\underline{v_1})$ as $m$ vectors in $k^n$, $(a_{1},\ldots,a_{m})\in k^{mn}$, where $a_j\in k^n$ for all $1\leq j \leq m$ and $a_1=(1,a'_1)$ with $a'_1\in k^{n-1}$.
  	Denote
  	\begin{equation*}
  	R_{m,n}(\underline{v_1})=\begin{pmatrix}
  	a_1\\\vdots \\a_{m}
  	\end{pmatrix}\in M_{m,n}(k).
  	\end{equation*}
  	Denote also $u_{1}(R_{m,n}(\underline{v_1})):=u_{1}(\underline{v_1})$.
  	The following lemma describes the orbit of the right action of $\mathrm{H}_{1,m-1}(k)\otimes\mathrm{H}_{1,n-1}(k)$ on the coset $	P_{mn-1,1}(k)w_1u_{1}\left(R_{m,n}(\underline{v_1})\right)$.
  	\begin{seclem}\label{tmn_orbit}
  		Let $h\in\mathrm{H}_{1,m-1}(k)$ and $g\in\mathrm{H}_{1,n-1}(k)$. Then,
  		\begin{equation}\label{tmn_orbit_eq}
  			P_{mn-1,1}(k)w_1u_{1}\left(R_{m,n}(\underline{v_1})\right)t(h,g)=P_{mn-1,1}(k)w_1u_{1}\left(h^TR_{m,n}(\underline{v_1})g\right).
  		\end{equation}
  	\end{seclem}
  	
  	\begin{proof}
  		By the Levi decomposition, we can write $h=h_Uh_\mathrm{S}$ and $g=g_Ug_\mathrm{S}$. We denote
  		\begin{equation*}
  		h_U:=\begin{pmatrix}
  		1&x\\0&I_{m-1}
  		\end{pmatrix},\qquad g_U:=\begin{pmatrix}
  		1&y\\0&I_{n-1}
  		\end{pmatrix},
  		\end{equation*}
  		where $x=(x_2,\ldots,x_m),\ y=(y_2,\ldots,y_n)$, and
  		$h_\mathrm{S}:=\mathrm{diag}(1,h')$, $g_\mathrm{S}:=\mathrm{diag}(1,g')$, where $h'=\left(h_{i,j}\right)_{2\leq i,j\leq m}\in \gl[k]{m-1}$, $g'=\left(g_{i,j}\right)_{2\leq i,j\leq n}\in \gl[k]{n-1}$.
		We have
  		\begin{equation}\label{mirab_levi_dec}
	  		t(h,g)=t\left(h_U,I_n\right)t\left(I_m,g_U\right)t\left(h_\mathrm{S},I_n\right)t\left(I_m,g_\mathrm{S}\right).
  		\end{equation}
  		
  		We first prove \cref{tmn_orbit_eq} for each one of the matrices on the right hand side of \cref{mirab_levi_dec}.
	  	
	  	\underline{Case $t(h_U,I_n)$}:
  		\begin{equation*}
  		u_{1}\left(R_{m,n}(\underline{v_1})\right)t(h_U,I_n)=\begin{pmatrix}A_1&A_2&\ldots&A_m\\&I_n&&\\&&\ddots&\\&&&I_n\end{pmatrix}\begin{pmatrix}I_n&x_2I_n&\ldots&x_mI_n\\&I_n&&\\&&\ddots&\\&&&I_n\end{pmatrix},
  		\end{equation*}
  		where
  		\begin{equation*}
			A_1= \begin{pmatrix}
				1&a'_1\\0&I_{n-1}
			\end{pmatrix}\in \gl[k]{n} \qquad \mathrm{and} \qquad \forall 2\leq i\leq m ,\  A_i=\begin{pmatrix}
			a_i\\0
			\end{pmatrix}\in M_n(k).
  		\end{equation*}
  		Thus,
  		\begin{equation*}
			u_{1}\left(R_{m,n}(\underline{v_1})\right)t(h_U,I_n)=\begin{pmatrix}A_1&x_2 A_1 + A_2&\ldots&x_mA_1+A_m\\&I_n&&\\&&\ddots&\\&&&I_n\end{pmatrix}=u_{1}\left(h_U^TR_{m,n}(\underline{v_1})\right).
		\end{equation*}
		
		\underline{Case $t((I_m,g_U)$}:
		\begin{equation*}
			u_{1}\left(R_{m,n}(\underline{v_1})\right)t(I_m,g_U)=\mathrm{diag}(I_n,g_U,\ldots,g_U)\begin{pmatrix}A_1g_U&A_2g_U&\ldots&A_mg_U\\&I_n&&\\&&\ddots&\\&&&I_n\end{pmatrix}.
		\end{equation*}
		By \Cref{uni_reps_claim} $\mathrm{diag}(I_n,g_U,\ldots,g_U)\in w_1^{-1}P_{mn-1,1}(k)w_1\cap U_{mn}(k)$. Therefore,
		\begin{equation*}
				P_{mn-1,1}(k)w_1u_{1}\left(R_{m,n}(\underline{v_1})\right)t(I_m,g_U)=	P_{mn-1,1}(k)w_1u_{1}\left(R_{m,n}(\underline{v_1})g_U\right).
		\end{equation*}

		\underline{Case $t(h_\mathrm{S},I_n)$}:
		\begin{equation*}
			u_{1}\left(R_{m,n}(\underline{v_1})\right)t(h_\mathrm{S},I_n)=\begin{pmatrix}A_1&A_2&\ldots&A_m\\&I_n&&\\&&\ddots&\\&&&I_n\end{pmatrix}\begin{pmatrix}I_n&0&\ldots&0\\0&h_{2,2}I_n&\ldots&h_{2,m}I_n\\\vdots&\vdots&\ddots&\vdots\\0&h_{m,2}I_n&&h_{m,m}I_n\end{pmatrix}.
		\end{equation*}
		We get that,
		\begin{equation*}
			u_{1}\left(R_{m,n}(\underline{v_1})\right)t(h_\mathrm{S},I_n)=t(h_\mathrm{S},I_n)\begin{pmatrix}A_1&\sum_{j=2}^m h_{j,2}A_j&\ldots&\sum_{j=2}^m h_{j,m}A_j\\&I_n&&\\&&\ddots&\\&&&I_n\end{pmatrix}.
		\end{equation*}
		Therefore,
		\begin{equation*}
			u_{1}\left(R_{m,n}(\underline{v_1})\right)t(h_\mathrm{S},I_n)=t(h_\mathrm{S},I_n)u_{1}\left(h_\mathrm{S}^TR_{m,n}(\underline{v_1})\right).
		\end{equation*}

		\underline{Case $t(I_m,g_\mathrm{S})$}:
		\begin{equation*}
		u_{1}\left(R_{m,n}(\underline{v_1})\right)t(I_m,g_\mathrm{S})=t(I_m,g_\mathrm{S})\begin{pmatrix}g_\mathrm{S}^{-1}A_1g_\mathrm{S}&g_\mathrm{S}^{-1}A_2g_\mathrm{S}&\ldots&g_\mathrm{S}^{-1}A_mg_\mathrm{S}\\&I_n&&\\&&\ddots&\\&&&I_n\end{pmatrix}.
		\end{equation*}
		Again, by the fact that the first column of $g_\mathrm{S}$ is $(1,0,\ldots,0)\in k^n$ and for $1\leq i\leq m$, only the first row of $A_i$ is nonzero, we have $g_\mathrm{S}^{-1}A_ig_\mathrm{S}=A_ig_\mathrm{S}$.
		Thus,
		\begin{equation*}
			u_{1}\left(R_{m,n}(\underline{v_1})\right)t(I_m,g_\mathrm{S})=t(I_m,g_\mathrm{S})u_{1}\left(R_{m,n}(\underline{v_1})g_\mathrm{S}\right).
		\end{equation*}
		
		Let $D$ be one of the matrices $t(h_\mathrm{S},I_n)$ or $t(I_m,g_\mathrm{S})$. Then, by the fact that the first column of $D$ is $(1,0,\ldots,0)\in k^{mn}$, we have $w_1Dw_1^{-1}\in P_{mn-1,1}(k)$. Hence,  \cref{tmn_orbit_eq} is true for the last two cases as well.
		
		All in all, together with \cref{mirab_levi_dec}, we get
			\begin{equation*}
				P_{mn-1,1}(k)w_1u_{1}\left(R_{m,n}(\underline{v_1})\right)t(h,g)=P_{mn-1,1}(k)w_1u_{1}\left(h_\mathrm{S}^Th_U^TR_{m,n}(\underline{v_1})g_Ug_\mathrm{S}\right),
			\end{equation*}
		and the lemma follows.
			
  	\end{proof}
  	
  	Next, we use \Cref{tmn_orbit} to show that for each $ \underline{v_1}\in k^{mn-1}$, the double coset $C_1(\underline{v_1})$ equals to one of the double cosets in \cref{split_to_double_coset}.
  	\begin{seclem}\label{shortning_lem}
  		There exists $0\leq r\leq n-1$ such that $C_1(\underline{v_1})=C_{(m-r)n}(\underline{b_r})$.
  	\end{seclem}
   	
   	\begin{proof}
   			Let $r:=\mathrm{rk}(R_{m,n}(\underline{v_1}))-1$, i.e. the dimension of the space spanned by $a_1,\ldots,a_m$ minus $1$. There exist $h_0\in \gl[k]{m}$ and $g_0\in \gl[k]{n}$ such that $h_0^TR_{m,n}(\underline{v_1})g_0=R_{m,n}(\underline{b'_r})$, where $\underline{b^\prime_r}$ corresponds to 
   			\begin{equation*}
   			R_{m,n}(\underline{b^\prime_r})=\begin{pmatrix}I_{r+1}&0\\0&0\end{pmatrix}.
   			\end{equation*}
   			In fact, in order to preserve the $1$ at the top left corner of $R_{m,n}(\underline{v_1})$, we must have  $h_0\in\mathrm{H}_{1,m-1}(k)$ and $g_0\in\mathrm{H}_{1,n-1}(k)$. Hence, by \Cref{tmn_orbit}, we have
   			\begin{equation*}
	   			C_1(\underline{v_1})=P_{mn-1,1}(k)w_1u_{1}\left(R_{m,n}(\underline{v_1})\right)t(h_0,g_0)T_{m,n}(k)=C_1(\underline{b'_r}).
   			\end{equation*}
   			Now, 
   			\begin{equation*}
   				C_1(\underline{b'_r})=P_{mn-1,1}(k)w_1E\left(E^{-1}u_{1}\left(\underline{b'_r}\right)E\right)T_{m,n}(k),
   			\end{equation*}
   			where
   			\begin{equation*}
				E:=t\left(\begin{pmatrix}
				&&1&\\&I_{m-r-2}&&\\1&&&\\&&&I_r\end{pmatrix},\begin{pmatrix}
				&&&1\\&&1&\\&\reflectbox{$\ddots$} &&\\1&&&\end{pmatrix}\right).
   			\end{equation*}
   			On the one hand, $E^{-1}u_{1}\left(\underline{b'_r}\right)E=u_{(m-r)n}\left(\underline{b_r}\right)$, and on the other, $P_{mn-1,1}(k)w_1E=P_{mn-1,1}(k)w_{(m-r)n}$.
   			Thus, $C_1(\underline{b'_r})=C_{(m-r)n}(\underline{b_r})$ as requested.	
  	\end{proof}

  	Now, we conclude that for each $1\leq j \leq mn$ and $\underline{v_j}\in k^{mn-j}$, the double coset  $C_j(\underline{v_j})$ equals to one of the double cosets in \cref{split_to_double_coset}.
 	\begin{seclem}\label{shortning_lem_gen}
 			Let $1\leq j \leq mn$ and $\underline{v_j}\in k^{mn-j}$. There exists $0\leq r\leq n-1$ such that $C_j(\underline{v_j})=C_{(m-r)n}(\underline{b_r})$.
 	\end{seclem}
  	\begin{proof}
  		If $j=1$ then by \Cref{shortning_lem} there exists $0\leq r\leq n-1$ such that $C_1(\underline{v_1})=C_{(m-r)n}(\underline{b_r})$ as requested. For all other cases  $1< j \leq mn$, we denote $j=\alpha n+\beta$, where $0\leq \alpha \leq m-1$ and $1\leq \beta \leq n$.
  		We have
  		\begin{equation*}
  			C_j(\underline{v_j})=P_{mn-1,1}(k)w_jE\left(E^{-1}u_{j}\left(\underline{v_j}\right)E\right)T_{m,n}(k),
  		\end{equation*}
  		where
  		\begin{equation*}
  			E:=t\left(\begin{pmatrix}
  			&I_{\alpha}&\\1&&\\&&I_{m-\alpha-1}\end{pmatrix},
  			\begin{pmatrix}&I_{\beta-1}&\\1&&\\&&I_{n-\beta}\end{pmatrix}\right).
  		\end{equation*}
  		On the one hand, $E^{-1}u_{j}\left(\underline{v_j}\right)E=u_{1}\left(\underline{v_1}\right)$, where $\underline{v_1}=(0,\underline{v_j})\in k^{mn-1}$, and on the other, $P_{mn-1,1}(k)w_jE=P_{mn-1,1}(k)w_{1}$.
  		Thus, $C_j(\underline{v_j})=C_{1}(\underline{v_1})$, and the proof follows from the $j=1$ case.
	\end{proof}   

	\Cref{shortning_lem_gen} covers all the possibilities for the double cosets that appear in \cref{second_split_to_double_coset}. i.e. 
	\begin{equation*}
	\bigcup \limits_{1\leq j \leq mn}\bigcup \limits_{\underline{v_j}\in k^{mn-j}}C_j(\underline{v_j})=  \bigcup\limits_{0\leq r \leq n-1}C_{(m-r)n}(\underline{b_r}).
	\end{equation*}
	
	It is left to show that the double cosets in \cref{split_to_double_coset} are pairwise disjoint. This is done in following lemma.
   \begin{seclem}\label{double_cosets_disjoint}
   		Let $0\leq r\neq \ell \leq n-1$ be two integers. Then, $C_{(m-r)n}(\underline{b_r})\neq C_{(m-\ell)n}(\underline{b_\ell})$.
   \end{seclem}
	\begin{proof}
		Assume that $C_{(m-r)n}(\underline{b_r}) = C_{(m-\ell)n}(\underline{b_\ell})$. i.e.,
		\begin{equation}\label{double_cosets_are_eq_first}
			P_{mn-1,1}(k)w_{(m-r)n}u_{(m-r)n}\left(\underline{b_r}\right)T_{m,n}(k) = P_{mn-1,1}(k)w_{(m-\ell)n}u_{(m-\ell)n}\left(\underline{b_\ell}\right)T_{m,n}(k).
		\end{equation}	
		We show that $r=\ell$. \Cref{double_cosets_are_eq_first} gives that there exist $p\in P_{mn-1,1}(k)$, $h\in \gl[k]{m}$, and $g\in \gl[k]{n}$ such that
		\begin{equation}\label{double_cosets_are_eq}
			w_{(m-\ell)n}^{-1}p w_{(m-r)n}u_{(m-r)n}\left(\underline{b_r}\right) = u_{(m-\ell)n}\left(\underline{b_\ell}\right)t(h,g).
		\end{equation}
		Denote 
		\begin{equation*}
			p:=\begin{pmatrix}
				A_1&A_2&\underline{y_1}\\A_3&A_4&\underline{y_2}\\0&0&d\end{pmatrix},
		\end{equation*}
		where $A_4$ is a $\ell n$ by $rn$ matrix and $d\in k^\times$.
		Then, the left hand side of \cref{double_cosets_are_eq} equals
		\begin{equation}\label{lhs_double_cosets_are_eq}
			w_{(m-\ell)n}^{-1}\begin{pmatrix}
			A_1&A_2&\underline{y_1}\\A_3&A_4&\underline{y_2}\\0&0&d\end{pmatrix} w_{(m-r)n}u_{(m-r)n}\left(\underline{b_r}\right) = \begin{pmatrix}
			A_1&\underline{y_1}&\underline{y_1}\cdot\underline{b_r} +A_2\\0&d&d\underline{b_r}\\A_3&\underline{y_2}&\underline{y_2}\cdot\underline{b_r} +A_4\end{pmatrix}.
		\end{equation}
		Denote $h=\left(h_{i,j}\right)_{1\leq i,j\leq m}$ and $g=\left(g_{i,j}\right)_{1\leq i,j\leq n}$. Then, the right hand side of \cref{double_cosets_are_eq} equals
		\begin{equation}\label{rhs_double_cosets_are_eq}
			u_{(m-\ell)n}\left(\underline{b_\ell}\right)t(h,g)= \begin{pmatrix}
			I_{(m-\ell)n-1}&0&0\\&1&\underline{b_\ell}\\ &&I_{\ell n}\end{pmatrix} \begin{pmatrix}
			h_{1,1}g&\cdots&h_{1,\ell}g\\\vdots&&\vdots\\h_{\ell,1}g&\cdots&h_{\ell,\ell}g
			\end{pmatrix}.
		\end{equation}
		
		We see that for $1\leq i \leq (m-r)n-1$ the coordinates $\left((m-r)n,i\right)$ of the matrix in the right hand side of \cref{lhs_double_cosets_are_eq} are all zero. Hence, by comparing \cref{lhs_double_cosets_are_eq,rhs_double_cosets_are_eq}, we get the following system of equations. For all $1\leq j \leq n$ and all $1\leq j'\leq m-r-1$
		\begin{equation*}
			\sum_{i=0}^\ell h_{m-\ell+i,j'}g_{n-i,j}=0.
		\end{equation*}
		The rows of $g$ are linearly independent, so we conclude that for $m-\ell \leq i \leq m$ and $1\leq j \leq m-r-1$ we have $h_{i,j}=0$. 
		This gives
		\begin{equation*}
			h=\begin{pmatrix}
				A&B\\0&D
			\end{pmatrix},
		\end{equation*}
	where $D\in M_{\ell+1,r+1}(k)$. i.e. $D$ has $\ell+1$ linearly independent rows in $k^{r+1}$. This implies $\ell\leq r$. From symmetry of $r$ and $\ell$ in \cref{double_cosets_are_eq_first} we get that $r\leq \ell$ as well. Thus, $r=\ell$ as requested.
	\end{proof}

	\subsection{The stabilizers}\label{stab_sec}

	In this section we compute $Q^{\varepsilon_r} = P_{mn-1,1}^{\varepsilon_r} \cap T_{m,n}$ for all $0\leq r < n$, where $\varepsilon_r$ are the different representatives in \cref{doubel_coset_reps_form}. We denote the following maximal parabolic subgroups of $\gl[k]{m}$ and $\gl[k]{n}$ by $P_{r+1}^m:=P_{m-r-1,r+1}(k)$ and $P_{r+1}^n:=P_{n-r-1,r+1}(k)$, respectively.
	Consider the following subgroup of $t\left(P_{r+1}^m(k), P_{r+1}^n(k)\right)\leq T_{m,n}(k)$,
	\begin{equation*}
	t_\Delta\left(P_{r+1}^m(k), P_{r+1}^n(k)\right):=\left\{t\left(\begin{pmatrix}
	A&B\\0&\lambda d^{\ast}\end{pmatrix},\begin{pmatrix}
	a&b\\0& d\end{pmatrix}\right)\in T_{m,n}(k)\bigg\vert d\in \gl[k]{r+1},\ \lambda\in k^\times\right\},
	\end{equation*}
	such that $d^{\ast}:=\tilde{w}_r^{-1}(d^{T})^{-1}\tilde{w}_r$, where
	\begin{equation*}
	\tilde{w}_r=\begin{pmatrix}
	&&1\\&\reflectbox{$\ddots$}&\\1&&
	\end{pmatrix}\in \gl[k]{r}.
	\end{equation*}
	
	\begin{prop}\label{stab_prop}
		Let $0\leq r \leq n-1$. Then, $Q^{\varepsilon_r}(k)=t_\Delta\left(P_{r+1}^m(k), P_{r+1}^n(k)\right)$.
	\end{prop}

	Before proving \Cref{stab_prop}, we need the following two lemmas. Generally, we have, 
\begin{seclem}\label{tensor_of_inv_id_lemma}
	Let $V,W$ be vector spaces of dimension $\ell$ over $k$, with bases $\{v_1,v_2,\ldots, v_\ell\}$ and $\{w_1,w_2,\ldots, w_\ell\}$, respectively.
	Let $B:V\times W\to k$ be a non-degenerate bilinear form, such that $B(v_i,w_j)=\delta_{i,j}$ for all $1\leq i,j\leq \ell$.
	Let $S:V\to V$ and $R:W\to W$ be linear transformations. Then, for $\alpha\in k$,
	\begin{equation}\label{tensor_of_inv_id}
		\left(S\otimes R\right)\left(\sum_{j=1}^{\ell}v_j\otimes w_j\right)=\alpha\sum_{j=1}^{\ell}v_j\otimes w_j
	\end{equation}
	iff $S \circ R^T = \alpha\cdot \mathrm{id}_V$. Therefore, in case $\alpha\neq 0$, \cref{tensor_of_inv_id} holds iff $S,R$ are invertible and $S=\alpha\left(R^T\right)^{-1}$, where $R^T : V \to V$ is the transformation adjoint to $R$ via the non-degenerate bilinear form $B$.
\end{seclem}

\begin{proof}
	 The bilinear form $B$ defines an isomorphism $\iota:V\otimes W\to \mathrm{Hom}_k(V,V)$ by sending $v\otimes w$ to the linear map $\iota(v\otimes w)=i_{v\otimes w}:V\to V$ defined by $i_{v\otimes w}(x)=B(x,w)v$. Under this isomorphism, the inverse image of $\mathrm{id}_V$ is $\sum_{j=1}^{\ell}v_j\otimes w_j$. We have $(S\circ i_{v\otimes w}) (x) = B (x,w) S(v) = i_{S(v)\otimes w}$, and $(i_{v\otimes w}\circ S) (x) = B (S(x),w) v = B(x, S^T(w))	v = i_{v\otimes S^T(w)}$. Thus, we get that $S\otimes \mathrm{id}_W$ corresponds to postcomposition by $S$, and that $\mathrm{id}_V \otimes R$ corresponds to precomposition by $R^T$.
	 Hence, \cref{tensor_of_inv_id} is equivalent under the isomorphism $\iota$ to $S\circ  R^T = \alpha\cdot \mathrm{id}_V$, as requested.
\end{proof}

We now make use of \Cref{tensor_of_inv_id_lemma}  to prove the following lemma,
\begin{seclem}\label{stab_lemma}
	Let $h\in\mathrm{GL}_m(k)$, $g\in\mathrm{GL}_n(k)$, and
	\begin{equation}\label{specific_ev}
	\underline{v}:=(0_n,\ldots,0_n,e_n^T,e_{n-1}^T,\ldots,e_{n-r}^T)
	\end{equation}
	where $0_n$ is the row vector of $n$ zeros (it appears $m-r-1$ times in $\underline{v}$).
	
	Then, there is $\lambda\in k^\times$ such that $\underline{v} t\left(h, g\right) = \lambda \underline{v}$ iff $t\left(h, g\right)\in t_\Delta\left(P_{r+1}^m(k), P_{r+1}^n(k)\right)$.
\end{seclem}
\begin{proof}
	Let $\{v_1,v_2,\ldots, v_m\}$ and $\{w_1,w_2,\ldots, w_n\}$ be the bases of standard row vectors of $k^m$ and $k^n$, respectively. 
	We take the basis of $k^m\otimes k^n$ in lexicographic order, i.e.
	\begin{equation}\label{lexic_basis}
	\left\{v_1\otimes w_1, \ldots, v_1\otimes w_n,\ldots,v_m\otimes w_1,\ldots, v_m\otimes w_n\right\}.
	\end{equation}
	We have 
	\begin{equation*}
	\underline{v}=\sum_{j=0}^{r}	v_{m-j}\otimes w_{n-r+j}.
	\end{equation*}
	
	Now, we write
	\begin{equation*}
	h=\begin{pmatrix}
	A&B\\C&D
	\end{pmatrix},\qquad 		g=\begin{pmatrix}
	a&b\\c&d
	\end{pmatrix},
	\end{equation*}
	where $D\in \mathrm{M}_{r+1}(k)$ and $d\in \mathrm{M}_{r+1}(k)$. Denote $E_0=\mathrm{Sp}\{v_{m-r},\ldots,v_m\}$, $E^\prime_0 = \mathrm{Sp}\{w_{n-r},\ldots,w_{n}\}$, $E_1=\mathrm{Sp}\{v_{1},\ldots,v_{m-r-1}\}$, and $E^{\prime}_1 = \mathrm{Sp}\{w_{1},\ldots,w_{n-r-1}\}$.
	The matrices $C,D$, and $c,d$ act from right on the space $k^{r+1}$. Let us denote the corresponding linear transformations $T_C$, $T_{c}$, $T_D$ and $T_{d}$, respectively, such that
	\begin{equation*}
	\begin{split}
	&T_C:E_0\to E_1,\qquad T_D:E_0\to E_0,\\
	&T_{c}:E^\prime_0\to E^\prime_1,\qquad T_{d}:E^\prime_0\to E^\prime_0.
	\end{split}
	\end{equation*}
	With this notation, the equation $\underline{v} t\left(h, g\right) = \lambda \underline{v}$ takes the form
	\begin{equation*}
	\sum_{j=0}^{r}	\left(T_C\left(v_{m-j}\right)+T_D\left(v_{m-j}\right)\right)\otimes \left(T_{c}\left(w_{n-r+j}\right)+T_{d}\left(w_{n-r+j}\right)\right) = 
	\lambda\sum_{j=0}^{r}	v_{m-j}\otimes w_{n-r+j}.
	\end{equation*}
	The tensor product is bilinear, $\underline{v}\in E_0\otimes E^\prime_0$ and we have
	\begin{equation}\label{system_of_condition_of_x}
	\begin{split}
	&\sum_{j=0}^{r} T_C\left(v_{m-j}\right)\otimes T_{c}\left(w_{n-r+j}\right) \in E_1\otimes E^\prime_1,\\
	&\sum_{j=0}^{r} T_C\left(v_{m-j}\right)\otimes T_{d}\left(w_{n-r+j}\right) \in E_1\otimes E^\prime_0,\\
	&\sum_{j=0}^{r} T_D\left(v_{m-j}\right)\otimes T_{c}\left(w_{n-r+j}\right) \in E_0\otimes E^\prime_1,\\
	&\sum_{j=0}^{r} T_D\left(v_{m-j}\right)\otimes T_{d}\left(w_{n-r+j}\right) \in E_0\otimes E^\prime_0.\\
	\end{split}
	\end{equation}
	We conclude, by comparing the coefficients of the basis, that \cref{system_of_condition_of_x} is equivalent to
	\begin{equation}\label{system_of_equations_of_x}
	\begin{cases}
	\sum_{j=0}^{r} T_C\left(v_{m-j}\right)\otimes T_{c}\left(w_{n-r+j}\right)=0,\\
	\sum_{j=0}^{r} T_C\left(v_{m-j}\right)\otimes T_{d}\left(w_{n-r+j}\right)=0,\\
	\sum_{j=0}^{r} T_D\left(v_{m-j}\right)\otimes T_{c}\left(w_{n-r+j}\right)=0,\\
	\sum_{j=0}^{r} T_D\left(v_{m-j}\right)\otimes T_{d}\left(w_{n-r+j}\right) = \lambda\sum_{j=0}^{r}	v_{m-j}\otimes w_{n-r+j}.
	\end{cases}
	\end{equation}
	We apply \Cref{tensor_of_inv_id_lemma} by taking $V=E_0$, $W=E^\prime_0$, $\ell =r+1$, $S=T_D$, $R=T_{d}$.
	We identify $E_0$ and $E^\prime_0$ with $k^{r+1}$ by the linear isomorphisms $\psi_{E_0}$ and $\psi_{E^\prime_0}$, defined by $\psi_{E_0}(v_{m-j})=e_{j+1}$ and $\psi_{E^\prime_0}(w_{n-r+j})=e_{j+1}$ for all $0\leq j\leq r$, where $\{e_1,\ldots,e_{r+1}\}$ is the standard basis of $k^{r+1}$.
	Then, $T_D(v)=\psi_{E_0}^{-1}(D\psi_{E_0}(v))$ for all $v\in E_0$ and $T_{d}(w)=\psi_{E^\prime_0}^{-1}(d \psi_{E^\prime_0}(w))$ for all $w\in E^\prime_0$.
	Now, $B$ in \Cref{tensor_of_inv_id_lemma} is given by
	\begin{equation*}
		B(v,w)=(\psi_{E_0}(v))^T \tilde{w}_r \psi_{E^\prime_0}(w),
	\end{equation*}
	where $v\in E_0$ and $w\in E^\prime_0$.
	Therefore, we have 
	\begin{equation*}
		B\left(T_D(v),w\right)=(\psi_{E_0}(v))^T D^T \tilde{w}_r \psi_{E^\prime_0}(w)=(\psi_{E_0}(v))^T\tilde{w}_r\tilde{w}_r^{-1} D^T \tilde{w}_r \psi_{E^\prime_0}(w).
	\end{equation*}
	The matrix $\tilde{D}:=\tilde{w}_r^{-1} D^T \tilde{w}_r$ defines a linear transformation $T_{\tilde{D}}$, such that, $T_{\tilde{D}}(w)=\psi_{E^\prime_0}^{-1}(\tilde{D} \psi_{E^\prime_0}(w))$ for all $w\in E^\prime_0$.
	Thus,
	\begin{equation*}
		B\left(T_D(v),w\right)=B\left(v,\psi_{E^\prime_0}^{-1}\left(\tilde{D} \psi_{E^\prime_0}(w)\right)\right)=B\left(v,T_{\tilde{D}}(w)\right).
	\end{equation*}
	
	Hence, $T_D^T(w)=T_{\tilde{D}}(w)$. By \Cref{tensor_of_inv_id_lemma}, the last equation in \cref{system_of_equations_of_x} is equivalent to $T_D\circ T_{d}^T=\lambda\mathrm{id}_{E_0}$ (by taking $\alpha=\lambda$), i.e. $T_D$ and $T_{d}$ are invertible, and $T_D=\lambda(T_{d}^T)^{-1}=\lambda T_{\tilde{d}}^{-1}$. So,
	\begin{equation*}
		D=\lambda\tilde{d}^{-1}=\lambda\tilde{w}_r^{-1}(d^{T})^{-1}\tilde{w}_r=\lambda d^{\ast},
	\end{equation*}
	as requested.
	Similarly, by \Cref{tensor_of_inv_id_lemma} (where we take $\alpha=0$), the second and the third equations in \cref{system_of_equations_of_x} are equivalent to $T_D\circ T_{d}^T=0$ and $T_D\circ T_{c}^T=0$. These last equations, in turn, are equivalent to say that both $T_C$ and $T_{c}$ are zero transformations, as $T_{d}$ and $T_D$ are invertible. i.e. $C=0$ and $c=0$, as requested.
	
\end{proof}

\begin{proof}[Proof of \Cref{stab_prop}]
	We first note that 
	\begin{equation*}
		Q^{\varepsilon_r} = P_{mn-1,1}^{\varepsilon_r} \cap T_{m,n}=\left(P_{mn-1,1} \cap T_{m,n}^{\varepsilon_r^{-1}}\right)^{\varepsilon_r}.
	\end{equation*}
	So, we find $P_{mn-1,1}(k) \cap T_{m,n}(k)^{\varepsilon_r^{-1}}$ and conjugate the result by $\varepsilon_r$. Let $X\in P_{mn-1,1}(k) \cap T_{m,n}(k)^{\varepsilon_r^{-1}}$. The condition $X\in  T_{m,n}(k)^{\varepsilon_r^{-1}}$ implies that there exist $h\in\mathrm{GL}_m(k)$ and $g\in\mathrm{GL}_n(k)$ such that 
	\begin{equation}\label{x_is_conj_tensor}
		X=\varepsilon_r t\left(h, g\right) \varepsilon_r^{-1}.
	\end{equation}
	On the other hand, the condition $X\in P_{mn-1,1}(k)$ yield that there exists $\lambda\in k^\times$ such that
	\begin{equation}\label{x_is_parab}
		e_{mn}^T X=\lambda e_{mn}^T,
	\end{equation}
	where $e_{mn}^T=(0,\ldots,0,1)$. Plugging \cref{x_is_conj_tensor} to \cref{x_is_parab} gives
	\begin{equation}\label{x_is_both}
		e_{mn}^T\varepsilon_r t\left(h, g\right)=\lambda e_{mn}^T\varepsilon_r.
	\end{equation}
	Recall that by \Cref{double_cosets_rep_thrm}, $\varepsilon_r=w_{(m-r)n}u_{(m-r)n}(\underline{b_r})$, where $\underline{b_r}:=\left(e_{n-1}^T,e_{n-2}^T,\ldots,e_{n-r}^T\right)$. Hence, \cref{x_is_both} becomes,
	\begin{equation}\label{x_is_both2}
		e_{mn}^T w_{(m-r)n}u_{(m-r)n}(\underline{b_r}) t\left(h, g\right)=\lambda e_{mn}^T w_{(m-r)n}u_{(m-r)n}(\underline{b_r}).
	\end{equation}
	Note that $e_{mn}^T w_{(m-r)n} = e_{(m-r)n}^T$ and that $e_{(m-r)n}^T u_{(m-r)n}(\underline{b_r})=\underline{v}$, where $\underline{v}$ is as in \cref{specific_ev}.
	Therefore, \cref{x_is_both2} can be written as
	\begin{equation*}
		\underline{v} t\left(h, g\right) = \lambda \underline{v}.
	\end{equation*}
	The proof now follows from \Cref{stab_lemma}.
\end{proof}

\subsection{Analysis of the contributions of the double cosets}\label{Only_one_cont_sec}
		
	We now plug \cref{unfold_g6} to our construction in \cref{eisen_def2}. By \Cref{double_cosets_rep_thrm}, we get that for $\Re(s)>>0$
	\begin{equation}\label{eisen_def3}
	\mathfrak{E}\left(f_{\omega_\pi,s},\varphi_{\pi}\right)(h)=\sum\limits _{r=0}^{n-1}\int \limits _{Z_n(\bba)\gl[k]{n}\backslash\gl{n}}\varphi_{\pi}(g) \sum \limits _{\gamma \in Q^{\varepsilon_r}(k)\backslash T_{m,n}(k)}  f_{\omega_\pi,s}\left(\varepsilon_r \gamma t\left(h,g\right)\right)dg.
	\end{equation}
	For all $0\leq r\leq n-1$ and all $n$, $\sum \limits _{\gamma}\left|  f_{\omega_\pi,s}\left(\varepsilon_r \gamma t\left(h,g\right)\right)\right|$ is of moderate growth in $g$, $\varphi_{\pi}$ is rapidly decreasing, and since $Z_n(\bba)\gl[k]{n}\backslash\gl{n}$ is of finite measure, the integral in \cref{eisen_def3} is absolutely convergent.
	In this section, we find the contribution of each double coset to the integral \cref{eisen_def3}. In \Cref{zero_cont_prop} we prove that all double cosets, except the one that corresponds to the open cell, contribute zero. We then analyze, in \Cref{const_is_eisen}, the contribution of the open cell, and show that it allows us to rewrite \cref{eisen_def3} in a form of an Eisenstein series on $\gl{m}$. In \Cref{left_p_trans_prop} we formally show that this series is the Eisenstein series on $\gl{m}$ corresponding to a section of $\mathrm{Ind}_{P_{m-n,n}(\bba)}^{\gl{m}}\left(1\otimes \tilde{\pi}\right) \delta_{P_{m-n,n}}^{s}$. In  \Cref{xi_is_section_prop} we state the last argument required for completing the proof of \Cref{main_thrm}.
	\begin{prop}\label{zero_cont_prop}
		Let $0\leq r \leq n-2$ and assume that $\Re(s)>>0$. Then,
			\begin{equation}\label{zero_cont_eq}
				\int \limits _{Z_n(\bba)\gl[k]{n}\backslash\gl{n}}\varphi_{\pi}(g) \sum \limits _{\gamma \in Q^{\varepsilon_r}(k)\backslash T_{m,n}(k)}  f_{\omega_\pi,s}\left(\varepsilon_r \gamma t\left(h,g\right)\right)dg=0.
			\end{equation}	
	\end{prop}
	\begin{proof}
	 	Recall that \Cref{stab_prop} gives 
	 	\begin{equation*}
	 		Q^{\varepsilon_r} =t_\Delta\left(P_{r+1}^m(k), P_{r+1}^n(k)\right)\leq t\left(P_{r+1}^m(k), P_{r+1}^n(k)\right).
	 	\end{equation*}
	 	Therefore, we split the summation in \cref{zero_cont_eq} as follows,
		\begin{equation}\label{zero_cont_eq_2}
			\sum \limits _{\gamma \in  t\left(P_{r+1}^m(k), P_{r+1}^n(k)\right)\backslash T_{m,n}(k)} \sum \limits _{\gamma^\prime \in Q^{\varepsilon_r}(k)\backslash t\left(P_{r+1}^m(k), P_{r+1}^n(k)\right)}  f_{\omega_\pi,s}\left(\varepsilon_r \gamma^\prime \gamma t\left(h,g\right)\right).
		\end{equation}	
		We can write the representatives of  $t\left(P_{r+1}^m(k), P_{r+1}^n(k)\right)\backslash T_{m,n}(k)$ as $\gamma=t\left(\gamma_1,\gamma_2\right)$, where $\gamma_1\in P_{r+1}^m(k)\backslash \gl[k]{m}$, and $\gamma_2\in P_{r+1}^n(k)\backslash \gl[k]{n}$. 
		Similarly, we can write the representatives of $Q^{\varepsilon_r}(k)\backslash t\left(P_{r+1}^m(k), P_{r+1}^n(k)\right)$ as $\gamma^\prime=t\left(\gamma_3, I_n\right)$, where
		\begin{equation*}
			\gamma_3=\begin{pmatrix}
		I_{m-r-1}&0\\0&d\end{pmatrix},
		\end{equation*}
		such that $d\in Z_{r+1}(k)\backslash\gl[k]{r+1}$.
	
		With this notation, we can write the left hand side of \cref{zero_cont_eq} as follows.
		\begin{equation}\label{zero_cont_eq2}
			\int \limits _{Z_n(\bba)\gl[k]{n}\backslash\gl{n}}\varphi_{\pi}(g) \sum \limits _{\gamma_1 }\sum \limits _{\gamma_2} \sum \limits _{\gamma_3} f_{\omega_\pi,s}\left(\varepsilon_r t\left(\gamma_3 \gamma_1  h, \gamma_2 g\right) \right)dg.
		\end{equation}	
		By the absolute convergence of the integral, we may switch the order of the $dg$-integral with the summations over $\gamma_1$ and $\gamma_3$. By the automorphic property of $\varphi_{\pi}$, we can write $\varphi_{\pi}(g)=\varphi_{\pi}(\gamma g)$ , which allows us to collapse the sum over $\gamma_2$ and the integral. Thus, \cref{zero_cont_eq2} becomes
		\begin{equation}\label{zero_cont_eq3}
			\sum \limits _{\gamma_1 }\sum \limits _{\gamma_3}  \int \limits _{Z_n(\bba)  P_{r+1}^n(k)\backslash \gl{n}}\varphi_{\pi}(g) f_{\omega_\pi,s}\left(\varepsilon_r t\left(\gamma_1 \gamma_3 h, g\right) \right)dg.
		\end{equation}	
		We now focus on the inner integral in \cref{zero_cont_eq3}. We unfold it using the Levi decomposition $P_{r+1}^n= \mathrm{S}_{r+1}^n\ltimes U_{r+1}^n$,
		\begin{equation*}
			\int \limits _{Z_n(\bba)\mathrm{S}_{r+1}^n(k)U_{r+1}^n(\bba)\backslash\gl{n}}\int \limits _{U_{r+1}^n(k)\backslash U_{r+1}^n(\bba)}\varphi_{\pi}(ug) f_{\omega_\pi,s}\left(\varepsilon_r t\left(\gamma_1 \gamma_3 h, ug\right) \right)dudg.
		\end{equation*}
		Notice that, $t\left(\gamma_1 \gamma_3 h, ug\right)=t\left(I_m,u\right)t\left(\gamma_1 \gamma_3 h, g\right)$. Now, $t\left(I_m,u\right)$ commutes with $\varepsilon_r$, as it is in the stabilizer $Q^{\varepsilon_r}$, and the section $f_{\omega_\pi,s}$ is invariant under left multiplication by unipotent elements. So, we can extract an inner integration
		\begin{equation}\label{cusp_int_vanish}
			\int \limits _{U_{r+1}^n(k)\backslash U_{r+1}^n(\bba)}\varphi_{\pi}(ug) du=0,
		\end{equation}
		which vanishes by the cuspidality of $\varphi_{\pi}$, for all $0\leq r\leq n-2$.
	\end{proof}

	\begin{prop}\label{const_is_eisen}
		Assume that $\Re(s)>>0$. Then,
		\begin{equation}\label{eisen_id_eq}
		\mathfrak{E}\left(f_{\omega_\pi,s},\varphi_{\pi}\right)(h) = \sum\limits_{\gamma\in P_{m-n,n}(k)\backslash\gl[k]{m}}\xi\left(f_{\omega_\pi,s},\varphi_{\pi}\right)(\gamma h),
		\end{equation}
		where
		\begin{equation}\label{global_section_f}
		\xi\left(f_{\omega_\pi,s},\varphi_{\pi}\right)( h) = \int \limits _{Z_n(\bba)\backslash\gl{n}}\varphi_{\pi}(g)  f_{\omega_\pi,s}\left(\tilde{\varepsilon}   t\left(h,g\right)\right)dg.
		\end{equation}
		The integral \cref{global_section_f} converges absolutely for $\Re(s)>>0$.
	\end{prop}
	\begin{proof}
		By now, \cref{eisen_def3} is simplified to
	\begin{equation}\label{eisen_def4}
		\mathfrak{E}\left(f_{\omega_\pi,s},\varphi_{\pi}\right)(h)=\int \limits _{Z_n(\bba)\gl[k]{n}\backslash\gl{n}}\varphi_{\pi}(g) \sum \limits _{\gamma \in Q^{\tilde{\varepsilon}}(k)\backslash T_{m,n}(k)}  f_{\omega_\pi,s}\left(\tilde{\varepsilon} \gamma t\left(h,g\right)\right)dg,
	\end{equation}
	where $\tilde{\varepsilon}:=\varepsilon_{n-1}$, and $Q^{\tilde{\varepsilon}}(k)=t_\Delta\left(P_{n}^m(k), \gl[k]{n}\right)$.
	As in the beginning of the proof of \Cref{zero_cont_prop}, we split the summation in \cref{eisen_def4} as follows,
	\begin{equation}\label{eisen_def4.1}
	\sum \limits _{\gamma \in  t\left(P_{n}^m(k), \gl[k]{n}\right)\backslash T_{m,n}(k)} \sum \limits _{\gamma^\prime \in Q^{\tilde{\varepsilon}}(k)\backslash t\left(P_{n}^m(k), \gl[k]{n}\right)}  f_{\omega_\pi,s}\left(\tilde{\varepsilon} \gamma^\prime \gamma t\left(h,g\right)\right).
	\end{equation}	
	As in the last proof, we can write the representatives of  $t\left(P_{n}^m(k), \gl[k]{n}\right)\backslash T_{m,n}(k)$ as $\gamma=t\left(\gamma_1,I_n\right)$, where $\gamma_1\in P_{n}^m(k)\backslash \gl[k]{m}$. 
	We can also write the representatives of $Q^{\tilde{\varepsilon}}(k)\backslash t\left(P_{n}^m(k), \gl[k]{n}\right)$ as $\gamma^\prime=t\left(\gamma_3, I_n\right)$, where
	\begin{equation*}
	\gamma_3=\begin{pmatrix}
	I_{m-n}&0\\0&d\end{pmatrix},
	\end{equation*}
	such that $d\in Z_{n}(k)\backslash\gl[k]{n}$. With this notation, we can write \cref{eisen_def4.1} as follows.
	\begin{equation}\label{eisen_def4.2}
	\int \limits _{Z_n(\bba)\gl[k]{n}\backslash\gl{n}}\varphi_{\pi}(g) \sum \limits _{\gamma_1 }\sum \limits _{\gamma_3} f_{\omega_\pi,s}\left(\tilde{\varepsilon} t\left(\gamma_3 \gamma_1  h,  g\right) \right)dg.
	\end{equation}
	Let 
	\begin{equation*}
	\left(\gamma_3^{-1}\right)^*:=\begin{pmatrix}
	I_{m-n}&0\\0&\left(d^{-1}\right)^*\end{pmatrix}.
	\end{equation*}
	Then, $	f_{\omega_\pi,s}$ is invariant under left multiplication by the element $ t\left(\gamma_3^{-1}, \left(\gamma_3^{-1}\right)^* \right)$, which also commutes with $\tilde{\varepsilon}$. So,
	\begin{equation*}
		f_{\omega_\pi,s}\left(\tilde{\varepsilon} t\left(\gamma_3 \gamma_1  h,  g\right) \right) = f_{\omega_\pi,s}\left(\tilde{\varepsilon} t\left(\gamma_1  h, \left(\gamma_3^{-1}\right)^* g\right) \right).
	\end{equation*}
	
	As in the last proof, we switch the order of the $dg$-integration with the summation over $\gamma_1$, and then collapse the sum over $\gamma_3$ and the integral.	
	Thus, \cref{eisen_def4.2} gets a form of an Eisenstein series. i.e.
	\begin{equation}\label{eisen_def4.3}
		\mathfrak{E}\left(f_{\omega_\pi,s},\varphi_{\pi}\right)(h) = \sum \limits _{\gamma_1\in P_{n}^m(k)\backslash \gl[k]{m}}  \int \limits _{Z_n(\bba) \backslash \gl{n}}\varphi_{\pi}(g) f_{\omega_\pi,s}\left(\tilde{\varepsilon} t\left(\gamma_1  h, g\right) \right)dg.
	\end{equation}	
	By writing $\gamma_1=\gamma$ we get the requested result.
	We note that for $\Re(s)>>0$, 
	\begin{equation*}
		\sum \limits _{\gamma_1\in P_{n}^m(k)\backslash \gl[k]{m}}  \int \limits _{Z_n(\bba) \backslash \gl{n}}\left|\varphi_{\pi}(g) f_{\omega_\pi,s}\left(\tilde{\varepsilon} t\left(\gamma_1  h, g\right) \right)\right|dg< \infty
	\end{equation*}	
	\end{proof}
	The integral in \cref{global_section_f} lies in the space of $ \mathrm{Ind}_{P_{m-n,n}(\bba)}^{\gl{m}}\left(1\otimes \tilde{\pi}\right) \delta_{P_{m-n,n}}^{s}$. The next proposition prove this fact at the formal level. The precise meaning of this integral is shown in \Cref{local_sec}.

	\begin{prop}\label{left_p_trans_prop}
		Assume that $\Re(s)>>0$. Let $p\in P_{m-n,n}(\bba)$ be of the form
		\begin{equation*}
			p=\begin{pmatrix}
			A&B\\&D\end{pmatrix},
		\end{equation*}
		where $A\in\gl{m-n},\ D\in\gl{n}$ and $B\in \mathrm{M}_{m-n,n}(\bba)$. Then
		\begin{equation}\label{left_p_trans_eq}
			\xi\left(f_{\omega_\pi,s},\varphi_{\pi}\right)\left(ph\right) = \delta_{P_{m-n,n}}^{s+\nicefrac{1}{2}}\left(p\right) \int\limits_{Z_n(\bba)\backslash\gl{n}}\varphi_{\pi}(D^\ast g)f_{\omega_\pi,s}\left(\tilde{\varepsilon} t\left(h,g\right)\right)dg.
		\end{equation}
	\end{prop}
	\begin{proof}
	We begin by the variables change $g\mapsto D^\ast g$.
	\begin{equation*}
		\xi\left(f_{\omega_\pi,s},\varphi_{\pi}\right)\left(ph\right) = \int\limits_{Z_n(\bba)\backslash	\gl{n}}\varphi_{\pi}(D^\ast g)f_{\omega_\pi,s}\left(\tilde{\varepsilon} t\left(ph,D^\ast g\right)\right) dg.
	\end{equation*}
	By the definition of $Q^{\tilde{\varepsilon}}$, the matrix $t\left(p,D^\ast \right)$ conjugated by $\tilde{\varepsilon}$, is an element in $P_{mn-1,1}$. Hence, we denote it by
	\begin{equation*}
	\begin{pmatrix}
	R&Y\\0&\alpha
	\end{pmatrix}:=\tilde{\varepsilon} t\left(p,D^\ast\right)\tilde{\varepsilon}^{-1},
	\end{equation*}
	where $R\in \gl{mn-1}$, $Y\in k^{mn-1}$, and $\alpha\in k^\times$.
	\begin{equation*}
	\begin{split}
		\xi\left(f_{\omega_\pi,s},\varphi_{\pi}\right)\left(ph\right) =&\omega_\pi ^{-1}(\alpha) \delta_{P_{mn-1,1}}^{s+\nicefrac{1}{2}}\left(\begin{pmatrix}
		R&Y\\0&\alpha
		\end{pmatrix}\right)\\&\cdot \int\limits_{Z_n(\bba)\backslash\gl{n}}\varphi_{\pi}(D^\ast g)f_{\omega_\pi,s}\left(\tilde{\varepsilon} t\left(h,g\right)\right)dg.
	\end{split}
	\end{equation*}
	By \cref{x_is_parab} we have $e_{mn}^T \tilde{\varepsilon} t\left(p,D^\ast\right)\tilde{\varepsilon}^{-1}=e_{mn}^T$. i.e., $\alpha e_{mn}^T=e_{mn}^T$. Thus, $\alpha=1$ and $\omega_\pi ^{-1}(\alpha) =1$.
	This also gives, $\left|\det(R)\right|=\left|\det\left(\tilde{\varepsilon} t\left(p,D^\ast\right)\tilde{\varepsilon}^{-1}\right)\right|$. Therefore,
	\begin{equation*}
	\begin{split}
		\delta_{P_{mn-1,1}}\left(\begin{pmatrix}
		R&Y\\0&\alpha
		\end{pmatrix}\right) =\left|\det(R)\right|&=\left|\det\left(p\right)\right|^n \left|\det\left(D^\ast\right)\right|^m\\& =\left| \det\left(A\right)\right|^n \left|\det\left(D\right)\right|^{-(m-n)}.
	\end{split}
	\end{equation*}
	This implies
	\begin{equation*}
		\delta_{P_{mn-1,1}}^{s+\nicefrac{1}{2}}\left(\tilde{\varepsilon} t\left(p,D^\ast\right)\tilde{\varepsilon}^{-1}\right) =\delta_{P_{m-n,n}}^{s+\nicefrac{1}{2}}\left(p\right).
	\end{equation*}
	The point now is that the function 
	\begin{equation*}
	D\mapsto \int\limits_{Z_n(\bba)\backslash	\gl{n}}\varphi_{\pi}(D^\ast g)f_{\omega_\pi,s}\left(\tilde{\varepsilon} t\left(h,g\right)\right) dg
	\end{equation*}
is formally a cusp form in the space of $\tilde{\pi}$, that is of the form $\varphi'_\pi (D^\ast)$ where $\varphi'_{\pi}$ is in the space of $\pi$.
	\end{proof}
	
	In order to complete the proof of \Cref{main_thrm}, we need the following proposition.
	\begin{prop}\label{xi_is_section_prop}
		The function on $\gl{m}$, $	\xi\left(f_{\omega_\pi,s},\varphi_{\pi}\right)$, defined for $\Re(s)$ sufficiently large by the integral \cref{global_section_f}, admits an analytic continuation to a
		meromorphic function of $s$ in the whole plane. It defines a smooth meromorphic
		section of
		\begin{equation}\label{ind_rep_eq}
			\rho_{\tilde{\pi},s}: = \mathrm{Ind}_{P_{m-n,n}(\bba)}^{\gl{m}}\left(1\otimes \tilde{\pi}\right) \delta_{P_{m-n,n}}^{s}.
		\end{equation}
		Thus, by \cref{eisen_id_eq}, $\mathfrak{E}\left(f_{\omega_\pi,s},\varphi_{\pi}\right)(h)$ is the Eisenstein series on $\gl{m}$, corresponding to the section $\xi\left(f_{\omega_\pi,s},\varphi_{\pi}\right)$ of $\rho_{\tilde{\pi},s}$.
	\end{prop}
	The proof of \Cref{xi_is_section_prop} is given at the end of \Cref{local_sec}, where we first make sense of the integral in \cref{global_section_f}.

\section{Local result} \label{local_sec}

We begin this section by explaining how the global integral in \cref{global_section_f_intro} gives rise to corresponding local integrals.
Then, in \Cref{unram_sec} we explicitly compute these local integrals at the unramified places.
In \Cref{bad_places_sec} we show that, in the \enquote{bad} places, Archimedean or ramified non-Archimedean, these local integrals are meromorphic functions (rational functions in $q_\nu^\pm s$ in the $p$-adic case).

In the course of proving \Cref{main_thrm}, we showed that, for $\Re(s)>>0$, $h\in\gl{n}$
\begin{equation*}
	\mathfrak{E}\left(f_{\omega_\pi,s},\varphi_{\pi}\right)(h)=\sum\limits_{\gamma\in P_{m-n,n}(k)\backslash\gl[k]{m}}\xi\left(f_{\omega_\pi,s},\varphi_{\pi}\right)(\gamma h).
\end{equation*}
where (for $\Re(s)>>0$),
\begin{equation}\label{global_section2}
	\xi\left(f_{\omega_\pi,s},\varphi_{\pi}\right)( h) = \int \limits _{Z_n(\bba)\backslash\gl{n}}\varphi_{\pi}(g)  f_{\omega_\pi,s}\left(\tilde{\varepsilon}   t\left(h,g\right)\right)dg.
\end{equation}
We have seen formally that $\xi\left(f_{\omega_\pi,s},\varphi_{\pi}\right)\in \mathrm{Ind}_{P_{m-n,n}(\bba)}^{\gl{m}}\left(1\otimes \tilde{\pi}\right) \delta_{P_{m-n,n}}^{s}$.
In order to substantiate this statement, we carry out a local study at each place of \cref{global_section2}. We denote by $\pi_{\nu}$ the generic, irreducible, smooth representation of $\gl[k_{\nu}]{n}$ (with central character $\omega_{\pi_\nu}$), which is the local factor of $\pi$ at $\nu$. 
We write $\pi$ as a restricted tensor product of the local representations $\pi_{\nu}$, $\pi = \ell_{\pi} (\otimes_\nu ^\prime \pi_{\nu})$, where $\ell_{\pi}$ is a fixed isomorphism. Assume that the cusp form $\varphi_{\pi}$ is decomposable, i.e. $\varphi_{\pi}=\ell_{\pi}(\otimes_\nu ^\prime v_{\pi_{\nu}})$, where $v_{\pi_{\nu}}$ lies in the space of $\pi_{\nu}$. So,
\begin{equation}\label{cusp_form_dec}
	\varphi_{\pi}(g)=\left(\pi(g)\varphi_{\pi}\right)(I_n)=\left(\pi(g)\ell_{\pi}\left(\otimes_\nu ^\prime v_{\pi_{\nu}}\right)\right)(I_n)=\ell_{\pi}\left(\otimes_\nu ^\prime \pi_{\nu}(g_\nu)v_{\pi_{\nu}}\right)(I_n),
\end{equation}
where $g=\sideset{}{^\prime}\prod_{\nu} g_\nu\in \gl{n}$.
Assume that $f_{\omega_{\pi},s}$ is decomposable and corresponds to a tensor product of local sections $f_{\omega_{\pi_\nu},s}$ of
\begin{equation}\label{local_ind_space_from_omega}
	\rho_{\omega_{\pi_{\nu}},s}:=\mathrm{Ind}_{P_{mn-1,1}(k_\nu)}^{\gl[k_\nu]{mn}}\left(1\otimes \omega_{\pi_{\nu}}^{-1}\right) \delta_{P_{mn-1,1}}^{s}.
\end{equation}

Let us fix a finite set of places $S_0$, containing the Archimedean places, for which, $\pi_{\nu}$ is unramified. For all $\nu\notin S_0$, we fix a spherical vector $v^\circ_{\pi_{\nu}}\in V_{\pi_\nu}$ and $f^\circ_{\omega_{\pi_\nu},s}$, unramified and normalized, such that for $p\in P_{mn-1,1}(k_\nu)$ and $x\in K_{n,\nu}$, $f^\circ_{\omega_{\pi_\nu},s}(I_n)=1$.

Now we can interpret \cref{global_section2} via a product over the places of $k$.
Our assumptions together with \cref{cusp_form_dec} give
\begin{equation}\label{global_section_dec1}
\begin{split}
	\xi\left(f_{\omega_\pi,s},\varphi_{\pi}\right)( h) &= \int \limits _{Z_n(\bba)\backslash\gl{n}}\ell_{\pi}\left(\underset{\nu}{\otimes}^\prime\pi_{\nu}(g_\nu)v_{\pi_{\nu}}\right)(I_n) \sideset{}{^\prime}\prod_{\nu} f_{\omega_{\pi_\nu},s}\left(\tilde{\varepsilon}t\left(h_\nu,g_\nu\right)\right)dg
	\\& = \int \limits _{Z_n(\bba)\backslash\gl{n}}\ell_{\pi}\left[\underset{\nu}{\otimes}^\prime f_{\omega_{\pi_\nu},s}\left(\tilde{\varepsilon}   t\left(h_\nu,g_\nu\right)\right)\pi_{\nu}(g_\nu)v_{\pi_{\nu}} \right](I_n) dg,
\end{split}
\end{equation}
where $h=\sideset{}{^\prime}\prod_{\nu} h_\nu$.
\Cref{global_section_dec1} is defined by the following limit.
\begin{equation}\label{global_section_dec2}
\begin{split}
	\xi\left(f_{\omega_\pi,s},\varphi_{\pi}\right)( h) = \lim\limits_{S_0\subseteq S\nearrow} \ell_{\pi}\Big[&\underset{\nu\in S}{\otimes}\int \limits _{Z_n(k_\nu)\backslash\gl[k_\nu]{n}}f_{\omega_{\pi_\nu},s}\left(\tilde{\varepsilon}   t\left(h_\nu,g_\nu\right)\right)\pi_{\nu}(g_\nu)v_{\pi_{\nu}}  dg_\nu \\&\bigotimes \underset{\nu\notin S}{\otimes}	\int \limits _{\ocal_{\nu}^\times\backslash K_{n,\nu} }f^\circ_{\omega_{\pi_\nu},s}\left(\tilde{\varepsilon}   t\left(h_\nu,g_\nu\right)\right)\pi_{\nu}(g_\nu)v^\circ_{\pi_{\nu}}  dg_\nu\Big](I_n),
\end{split}
\end{equation}
where the limit is taken over increasing finite sets of places $S\supseteq S_0$.
For all $g_\nu\in K_{n,\nu}$ we have 
\begin{equation*}
	f^\circ_{\omega_{\pi_\nu},s}\left(\tilde{\varepsilon}   t\left(h_\nu,g_\nu\right)\right)= f^\circ_{\omega_{\pi_\nu},s}\left(\tilde{\varepsilon}   t\left(h_\nu,I_n\right)\right).
\end{equation*}
This gives
\begin{equation*}
	\int \limits _{\ocal_{\nu}^\times\backslash K_{n,\nu} }f^\circ_{\omega_{\pi_\nu},s}\left(\tilde{\varepsilon}   t\left(h_\nu,g_\nu\right)\right)\pi_{\nu}(g_\nu)v^\circ_{\pi_{\nu}}  dg_\nu=v^\circ_{\pi_{\nu}}  f^\circ_{\omega_{\pi_\nu},s}\left(\tilde{\varepsilon}   t\left(h_\nu,I_n\right)\right).
\end{equation*}
Therefore, we can rewrite \cref{global_section_dec2} as
\begin{equation}\label{global_section_dec3}
\begin{split}
\xi\left(f_{\omega_\pi,s},\varphi_{\pi}\right)( h) = \lim\limits_{S_0\subseteq S\nearrow} \ell_{\pi}\Big[&\underset{\nu\in S}{\otimes}\int \limits _{Z_n(k_\nu)\backslash\gl[k_\nu]{n}}f_{\omega_{\pi_\nu},s}\left(\tilde{\varepsilon}   t\left(h_\nu,g_\nu\right)\right)\pi_{\nu}(g_\nu)v_{\pi_{\nu}}  dg_\nu \\&\bigotimes \underset{\nu\notin S}{\otimes}f^\circ_{\omega_{\pi_\nu},s}\left(\tilde{\varepsilon}   t\left(h_\nu,I_n\right)\right)v^\circ_{\pi_{\nu}}  \Big](I_n).
\end{split}
\end{equation}
All in all, the inner local integrals in \cref{global_section_dec3} are the corresponding local integrals to our global construction. We denote them by 
\begin{equation}\label{local_int}
	I\left(f_{\omega_{\pi_\nu},s},v_{\pi_{\nu}}\right)(h_\nu)=\int \limits _{Z_n(k_\nu)\backslash\gl[k_\nu]{n}} f_{\omega_{\pi_\nu},s}\left(\tilde{\varepsilon}   t\left(h_\nu,g_\nu\right)\right)\pi_{\nu}(g_\nu)v_{\pi_{\nu}} dg_\nu.
\end{equation}

In this section we use the Cartan decomposition (modulo the center). We can write $g_\nu\in Z_n(k_\nu)\backslash\gl[k_{\nu}]{n}$ as 
 \begin{equation}\label{cartan_dec}
 g_\nu=A \underline{t}B,
 \end{equation}
 where $A,B\in K_{n,\nu}$ and $\underline{t}=\mathrm{diag}\left(t_1,\ldots,t_{n-1},1\right)$ such that for $\nu<\infty$,
 \begin{equation*}
 \underline{t}\in Z_n(k_\nu)\backslash T^-:=\left\{\mathrm{diag}\left(\pomega^{r_1},\ldots,\pomega^{r_{n-1}},1\right)\in \gl[k_\nu]{n}|\ r_1\geq\ldots \geq r_{n-1}\geq 0 \right\}
 \end{equation*}
 and for $\nu=\infty$,
  \begin{equation*}
 \underline{t}\in Z_n(k_\nu)\backslash T^-:=\left\{\mathrm{diag}\left(e^{r_1},\ldots,e^{r_{n-1}},1\right)\in \gl[k_\nu]{n}|\ r_1\geq\ldots \geq r_{n-1}\geq 0 \right\}.
 \end{equation*}
 In both cases $\underline{t}$ satisfies $|t_1|\leq\ldots\leq |t_{n-1}|\leq 1$.
 
 We also make use of the following integral formula. We refer to the paper \cite{ichino2010periods} for the statement of this claim
 in this explicit form.
 \begin{claim}[\cite{ichino2010periods} p.12]\label{kak_integral}
 	Let $G$ be a reductive algebraic group defined over $k_\nu$, and $G = KT^-K$ its Cartan decomposition. Let $f\in L^1(G)$. Then,
 	\begin{equation*}
 	\int\limits _{G} f(g)dg=\int\limits _{T^-} \mu(a)\int\limits_{K^2} f(k_1ak_2)dk_1dadk_2,
 	\end{equation*}
 	where $\mu(a)=\mu(KaK)/\mu(K)$ for non-Archimedean places $\nu$ and $\mu(a)\geq 0$ for the Archimedean places.
 \end{claim}
 Moreover, we have the following formulae. By \cite[Chapter V, eq. (2.9)]{macdonald1998symmetric},
 \begin{claim}\label{jacobian_of_kak_int_non_arch}
 	Let $\nu<\infty$. Let $\underline{t}=\mathrm{diag}\left(\pomega^{r_1},\ldots,\pomega^{r_{n}}\right)\in T^-$.
 	Suppose that the integers $n_i$ are the lengths of constant runs in the sequence $(r_i)$, so that $n_1 +\ldots  + n_\ell = n$ 	and	
 	\begin{equation*}
 		r_1 =\ldots  = r_{n_1} > r_{n_1+1} = \ldots = r_{n_1+n_2} > r_{n_1+n_2+1} =\ldots
 	\end{equation*}
 	Then,
 	\begin{equation*}
 		\mu(\underline{t})=q_\nu ^{\sum_{i=1}^n(n-2i+1)r_i} \frac{\varphi_n(q^{-1})}{(1-q^{-1})^n} \prod_{j=1}^\ell \frac{(1-q^{-1})^{n_j}}{\varphi_{n_j}(q^{-1})},
	\end{equation*}
 	where as usual $\varphi_j(t)=(1-t)(1-t^2)\cdots(1-t^j)$.
 \end{claim}
 By \cite[Section I.5, Theorem 5.8]{helgason1984groups} or by \cite[Proposition 5.28]{knapp2001representation},
 \begin{claim}\label{jacobian_of_kak_int_arch}
 	Let $\nu=\infty$. Let $a\in T^-$. Set $H$ in the Lie algebra of $T^-$ such that $a=e^H$. Then,
 	\begin{equation*}
 		\mu(a)= \sum_{\alpha\in \Sigma _+}\left| \sinh \alpha(H)\right|^{\dim \mathfrak{g}_\alpha}.
	\end{equation*}
 \end{claim}

\subsection{Unramified computation}\label{unram_sec}
In this section we prove \Cref{unrm_calc_GJ_thrm} by computing the local integrals \cref{local_int} at the unramified places $\nu$. i.e. 
\begin{equation}\label{local_int_unr}
I\left(f_{\omega_{\pi_\nu},s}^\circ,v^\circ_{\pi_{\nu}}\right)(h_\nu)=\int \limits _{Z_n(k_\nu)\backslash\gl[k_\nu]{n}}f_{\omega_{\pi_\nu},s}^\circ\left(\tilde{\varepsilon}   t\left(h_\nu,g_\nu\right)\right)\pi_{\nu}(g_\nu)v^\circ_{\pi_{\nu}}  dg_\nu.
\end{equation}

As part of the proof of \Cref{unrm_calc_GJ_thrm}, we also conclude by the end of this section that
\begin{seclem}\label{local_is_unr}
	$I\left(f_{\omega_{\pi_\nu},s}^\circ,v^\circ_{\pi_{\nu}}\right)$ is an unramified section of $\mathrm{Ind}_{P_{m-n,n}(k_\nu)}^{\gl[k_\nu]{m}}\left(1\otimes \tilde{\pi}_\nu\right) \delta_{P_{m-n,n}}^{s}$.
\end{seclem}
In the meantime, we note that $I\left(f_{\omega_{\pi_\nu},s}^\circ,v^\circ_{\pi_{\nu}}\right)$ is left-invariant under $K_{m,\nu}$ by the fact that $f_{\omega_{\pi_\nu},s}^\circ$ is unramified. In addition, a similar (local) computation to \Cref{left_p_trans_prop} formally gives that  $I\left(f_{\omega_{\pi_\nu},s}^\circ,v^\circ_{\pi_{\nu}}\right)$ lies in   $\mathrm{Ind}_{P_{m-n,n}(k_\nu)}^{\gl[k_\nu]{m}}\left(1\otimes \tilde{\pi}_\nu\right) \delta_{P_{m-n,n}}^{s}$. Therefore, by writing the Iwasawa decomposition $h_\nu=h_P h_K$, where $h_P\in P_{m-n,n}(k_\nu)$ and $h_K\in K_{m}$, we get that it is sufficient to evaluate $I\left(f_{\omega_{\pi_\nu},s}^\circ,v^\circ_{\pi_{\nu}}\right)$ at $I_m$.

Before proving \Cref{unrm_calc_GJ_thrm}, we find in \Cref{local_int_cartan} and \Cref{sec_val_on_tor} the value of $f_{\omega_{\pi_\nu},s}^\circ\left(\tilde{\varepsilon}   t\left(I_m,g_\nu\right)\right)$, as an expression of $\underline{t}$.

\begin{seclem}\label{local_int_cartan}
	Let  $g_\nu=A \underline{t} B$, be the Cartan decomposition (modulo the center) of $g_\nu$ as in \cref{cartan_dec}. Then,
	\begin{equation*}
	f_{\omega_{\pi_\nu},s}^\circ\left(\tilde{\varepsilon}   t\left(I_m,g_\nu\right)\right)=	f_{\omega_{\pi_\nu},s}^\circ\left(\tilde{\varepsilon} t\left(I_m, \underline{t} \right)\right).
	\end{equation*}
	
\end{seclem}
\begin{proof}
The section $f_{\omega_{\pi_\nu},s}^\circ$ is unramified, so, in particular, it is right-invariant under $t(\mathrm{diag}(I_{m-n},x^*),I_n)$, where $x\in K_{n,\nu}$. Thus,
\begin{equation*}
	f_{\omega_{\pi_\nu},s}^\circ\left(\tilde{\varepsilon}   t\left(I_m,xg_\nu\right)\right)=f_{\omega_{\pi_\nu},s}^\circ\left(\tilde{\varepsilon}   t\left(\mathrm{diag}(I_{m-n},x^*),xg_\nu\right)\right).
\end{equation*}
On the other hand, $t(\mathrm{diag}(I_{m-n},x^*),x)\in Q^{\tilde{\varepsilon}}(k_\nu)$, so its conjugation with $\tilde{\varepsilon}$ lies in $P_{mn-1,1}(k_\nu)\cap K_{mn,\nu}$. Since $f_{\omega_{\pi_\nu},s}^\circ$ is left invariant under  $P_{mn-1,1}(k_\nu)\cap K_{mn,\nu}$, we conclude that 
\begin{equation}\label{local_int_k_left_inv}
	f_{\omega_{\pi_\nu},s}^\circ\left(\tilde{\varepsilon}   t\left(I_m,xg_\nu\right)\right)=	f_{\omega_{\pi_\nu},s}^\circ\left(\tilde{\varepsilon}   t\left(I_m,g_\nu\right)\right).
\end{equation}

We conclude that the function $g_\nu\mapsto	f_{\omega_{\pi_\nu},s}^\circ\left(\tilde{\varepsilon}   t\left(I_m,g_\nu\right)\right)$ is bi-invariant under $K_{n,\nu}$. Therefore,
\begin{equation*}
	f_{\omega_{\pi_\nu},s}^\circ\left(\tilde{\varepsilon}   t\left(I_m,g_\nu\right)\right)=	f_{\omega_{\pi_\nu},s}^\circ\left(\tilde{\varepsilon}   t\left(I_m,A \underline{t} B\right)\right)=	f_{\omega_{\pi_\nu},s}^\circ\left(\tilde{\varepsilon} t\left(I_m, \underline{t} \right)\right).
\end{equation*}
\end{proof}

\begin{prop}\label{sec_val_on_tor}
	In the notation above,
	\begin{equation*}
		f_{\omega_{\pi_\nu},s}^\circ\left(\tilde{\varepsilon}   t\left(I_m,\underline{t}\right)\right)=\left|\det\underline{t}\right|^{ms+\frac{m}{2}}.
	\end{equation*}
\end{prop}

\begin{proof}
	We denote $\underline{t}^\Delta:=t\left(I_m, \underline{t} \right)$.
	Recall that $\tilde{\varepsilon}=\tilde{w}\tilde{u}$ where we write $\tilde{w}:=w_{(m-n+1)n}$ and $\tilde{u}:=u_{(m-n+1)n}(\underline{b_{n-1}})$.
	In this notation we have,
	\begin{equation*}
		\tilde{\varepsilon} \underline{t}^\Delta=\left(\tilde{w}\underline{t}^\Delta\tilde{w}^{-1}\right)\tilde{w}\left(\left(\underline{t}^\Delta\right)^{-1}\tilde{u}\underline{t}^\Delta\right)=\left(\tilde{w}\underline{t}^\Delta\tilde{w}^{-1}\right)\tilde{w}u_{(m-n+1)n}(\underline{b_{n-1}}(\underline{t})),
	\end{equation*}
	where $\underline{b_{n-1}}(\underline{t})=\left(t_{n-1}\cdot e_{n-1}^T,\ldots,t_1\cdot e_{1}^T\right)$. All the elements above the diagonal of the unipotent matrix $u_{(m-n+1)n}(\underline{b_{n-1}}(\underline{t}))$ are bounded in absolute value by $1$. Therefore, $\tilde{w}u_{(m-n+1)n}(\underline{b_{n-1}}(\underline{t}))\in K_{mn,\nu}$, which implies
	\begin{equation}\label{conj_torus_by_perm}
		f_{\omega_{\pi_\nu},s}^\circ\left(\tilde{\varepsilon}  
		\right)=f_{\omega_{\pi_\nu},s}^\circ \left(\tilde{w}\underline{t}^\Delta\tilde{w}^{-1}\right).
	\end{equation}
	Since $\tilde{w}\underline{t}^\Delta\tilde{w}^{-1}$ equals  $\underline{t}^\Delta$ up to the order of the elements on the diagonal
	\begin{equation}\label{section_on_torus}
		f_{\omega_{\pi_\nu},s}^\circ \left(\tilde{w}\underline{t}^\Delta\tilde{w}^{-1}\right)=\omega^{-1}_{\pi_{\nu}}(\alpha)\delta_{P_{mn-1,1}}^{s+\frac{1}{2}}\left(\tilde{w}\underline{t}^\Delta\tilde{w}^{-1}\right),
	\end{equation}
	where $\alpha$ is the $(mn,mn)$-th coordinate of $\tilde{w}\underline{t}^\Delta\tilde{w}^{-1}$. It is given by
	\begin{equation*}
		e_{mn}^T\left(\tilde{w}\underline{t}^\Delta\tilde{w}^{-1}\right)e_{mn}=e_{(m-n+1)n}^T\underline{t}^\Delta e_{(m-n+1)n}.
	\end{equation*}
	i.e. it is the $((m-n+1)n,(m-n+1)n)$-th coordinate of $\underline{t}^\Delta$, which equals $1$. So $\alpha=1$ and
	\begin{equation}\label{omega_is_one}
		\omega^{-1}_{\pi_{\nu}}(\alpha)=1.
	\end{equation}
	This also gives that
	\begin{equation}\label{delta_is_det}
		\delta_{P_{mn-1,1}}\left(\tilde{w}\underline{t}^\Delta\tilde{w}^{-1}\right)=\left|\det \underline{t}^\Delta\right| = \left|\det \underline{t}\right|^m=\prod_{i=1}^{n-1}\left|t_i\right|^m.
	\end{equation}
	We now plug \cref{omega_is_one} and \cref{delta_is_det} in \cref{section_on_torus}. The proposition now follows by plugging this result in \cref{conj_torus_by_perm}.
\end{proof}

We are now ready to prove \Cref{unrm_calc_GJ_thrm}.
\begin{proof}[Proof of \Cref{unrm_calc_GJ_thrm}]
	We normalize the measure such that 
	\begin{equation}\label{k_mes_is_1}
		\mu(K_{n,\nu})=\int\limits_{K_{n,\nu}}dx=1.
	\end{equation}
	Let  $g_\nu=A \underline{t} B$, be the Cartan decomposition (modulo the center) of $g_\nu$ as in \cref{cartan_dec}.
	By \Cref{kak_integral} we can write \cref{local_int_unr} as
	\begin{equation}\label{local_int_unr_split0}
	I\left(f_{\omega_{\pi_\nu},s}^\circ,v^\circ_{\pi_{\nu}}\right)(I_m)=\int\limits _{Z_n(k_\nu)\backslash T^-}\mu(\underline{t}) \int\limits_{K_{n,\nu}^2}f_{\omega_{\pi_\nu},s}^\circ\left(\tilde{\varepsilon}   t\left(I_m,A\underline{t}B\right)\right)\pi_{\nu}(A\underline{t}B)v^\circ_{\pi_{\nu}}  dAdBd\underline{t},
	\end{equation}
	where $\mu(\underline{t})=\mu(K_{n,\nu}\underline{t}K_{n,\nu})$.
	We assume that $\Re(s)>>0$. By \Cref{local_int_cartan} we have $f_{\omega_{\pi_\nu},s}^\circ\left(\tilde{\varepsilon}   t\left(I_m,A\underline{t}B\right)\right)=f_{\omega_{\pi_\nu},s}^\circ\left(\tilde{\varepsilon}   t\left(I_m,\underline{t}\right)\right)$. In addition, $v^\circ_{\pi_{\nu}}$ is spherical so $\pi_{\nu}(A\underline{t}B)v^\circ_{\pi_{\nu}}=\pi_{\nu}(A\underline{t})v^\circ_{\pi_{\nu}}$.
	Therefore, \cref{local_int_unr} equals
	\begin{equation}\label{local_int_unr_split}
	I\left(f_{\omega_{\pi_\nu},s}^\circ,v^\circ_{\pi_{\nu}}\right)(I_m)=\int \limits _{Z_n(k_\nu)\backslash T^-}\mu(\underline{t})f_{\omega_{\pi_\nu},s}^\circ\left(\tilde{\varepsilon}   t\left(I_m,\underline{t}\right)\right)\int\limits_{K_{n,\nu}}\pi_{\nu}(A\underline{t})v^\circ_{\pi_{\nu}}  dAd\underline{t}.
	\end{equation}
	
	Applying \Cref{sec_val_on_tor} gives
	\begin{equation}\label{local_int_unr_split2}
		I\left(f_{\omega_{\pi_\nu},s}^\circ,v^\circ_{\pi_{\nu}}\right)(I_m)=\int \limits _{Z_n(k_\nu)\backslash T^-}\mu(\underline{t})\left|\det \underline{t} \right|^{ms+\frac{m}{2}}\int\limits_{K_{n,\nu}}\pi_{\nu}(A\underline{t})v^\circ_{\pi_{\nu}}dA d\underline{t}.
	\end{equation}
	The inner integral is an unramified functional on $V_\pi$. Hence, there exists a constant $c(\underline{t})$ such that
	\begin{equation}\label{int_over_k_is_unr}
		\int\limits_{K_{n,\nu}}\pi_{\nu}(A\underline{t})v^\circ_{\pi_{\nu}}  dA=c(\underline{t})v^\circ_{\pi_{\nu}}.
	\end{equation}
	We apply the unique vector $\check{v}^\circ_{\tilde{\pi}_{\nu}}\in V_{\tilde{\pi}}$, such that $\left<v_{\pi_{\nu}}^\circ,\check{v}_{\pi_{\nu}}^{\circ}\right>=1$, on \cref{int_over_k_is_unr} 
	\begin{equation*}
		c(\underline{t})=\int\limits_{K_{n,\nu}}\left<\pi_{\nu}(A\underline{t})v^\circ_{\pi_{\nu}},\check{v}_{\pi_{\nu}}^{\circ}\right>  dA=
		\int\limits_{K_{n,\nu}}\left<\pi_{\nu}(\underline{t})v^\circ_{\pi_{\nu}},\tilde{\pi}_{\nu}(A^{-1})\check{v}_{\pi_{\nu}}^{\circ}\right>dA=\left<\pi_{\nu}(\underline{t})v^\circ_{\pi_{\nu}},\check{v}_{\pi_{\nu}}^{\circ}\right>.
	\end{equation*}	
	Therefore, by \cref{k_mes_is_1} we conclude that $c(\underline{t})= c_{v_{\pi_{\nu}}^\circ,\check{v}_{\pi_{\nu}}^{\circ}}(\underline{t})$ (the matrix coefficient of $\pi_\nu$).

	Thus, for $\Re(s)>>0$
	\begin{equation}\label{local_int_unr_split3}
		I\left(f_{\omega_{\pi_\nu},s}^\circ,v^\circ_{\pi_{\nu}}\right)(I_m)=\int \limits _{Z_n(k_\nu)\backslash T^-}\mu(\underline{t})c_{v_{\pi_{\nu}}^\circ,\check{v}_{\pi_{\nu}}^{\circ}}(\underline{t}){v}^\circ_{{\pi}_{\nu}}\left|\det \underline{t}\right|^{ms+\frac{m}{2}}d\underline{t},
	\end{equation}
	By the fact that $	I\left(f_{\omega_{\pi_\nu},s}^\circ,v^\circ_{\pi_{\nu}}\right)(I_m)$ is a $K_{n,\nu}$-invariant vector, there exists a constant $c(I,s)$ such that
	\begin{equation}\label{local_int_unr_split4}
		I\left(f_{\omega_{\pi_\nu},s}^\circ,v^\circ_{\pi_{\nu}}\right)(I_m)=c(I,s){v}^\circ_{{\pi}_{\nu}}.
	\end{equation}
	Applying $\check{v}^\circ_{\tilde{\pi}_{\nu}}$ on \cref{local_int_unr_split4} gives
	\begin{equation}\label{local_int_unr_split5}
		c(I,s)=\int \limits _{Z_n(k_\nu)\backslash T^-}\mu(\underline{t})c_{v_{\pi_{\nu}}^\circ,\check{v}_{\pi_{\nu}}^{\circ}}(\underline{t})\left|\det \underline{t}\right|^{ms+\frac{m}{2}}d\underline{t}.
	\end{equation}
	
	We now consider the Godement-Jacquet zeta integral \cref{GJ_zeta}. We evaluate it at the normalized matrix coefficient, which correspond to the spherical vectors $v^\circ_{\pi_{\nu}},\check{v}_{\pi_{\nu}}^\circ$, and at $\Phi_0$, the characteristic function of $M_n\left(\ocal_{\nu}\right)$,
	\begin{equation}\label{GJ_at_spheric}
		Z_{\mathrm{GJ}}\left(s,c_{v_{\pi_{\nu}}^\circ,\check{v}_{\pi_{\nu}}^{\circ}},\Phi_0\right)=\int\limits_{\gl[k_\nu]{n}}\Phi_0(g_\nu)c_{v_{\pi_{\nu}}^\circ,\check{v}_{\pi_{\nu}}^{\circ}}(g_\nu)\left|\det g_\nu\right|^{s+\frac{n-1}{2}}dg_\nu.
	\end{equation}
	Notice that the integrand is bi-invariant under $K_{n,\nu}$. Therefore, by writing the Cartan decomposition (modulo the center) of $g_\nu$ as in \cref{cartan_dec}, applying \Cref{kak_integral} on \cref{GJ_at_spheric}, and splitting the integral through the center, we have
	\begin{equation}\label{GJ_at_spheric_split}
		Z_{\mathrm{GJ}}\left(s,c_{v_{\pi_{\nu}}^\circ,\check{v}_{\pi_{\nu}}^{\circ}},\Phi_0\right)=\int\limits_{Z_n(k_\nu)\backslash T^-}\mu(\underline{t})F_{s,\Phi_0}(\underline{t})c_{v_{\pi_{\nu}}^\circ,\check{v}_{\pi_{\nu}}^{\circ}}(\underline{t})\left|\det \underline{t}\right|^{s+\frac{n-1}{2}}d\underline{t},
	\end{equation}
	where
	\begin{equation}\label{GJ_center_int}
		F_{s,\Phi_0}(\underline{t}):=\int\limits_{k_\nu^\times}\Phi_0(a \underline{t})|a|^{n\left(s+\frac{n-1}{2}\right)} \omega_{\pi_{\nu}}(a)d^\times a.
	\end{equation}
	For $a\in k^\times$, 
	\begin{equation*}
	a\underline{t} \in M_n\left(\ocal_{\nu}\right) \iff |a|\leq |1,
	\end{equation*}
	i.e. 
	\begin{equation}\label{GJ_schwartz}
		\Phi_0(a \underline{t})=\begin{cases}
						1,&|a|\leq 1\\
						0,&|a|>1.	
		\end{cases}
	\end{equation}
	By \cref{GJ_schwartz}, we find that \cref{GJ_center_int} equals
	\begin{equation}\label{GJ_over_center}
		F_{s,\Phi_0}(\underline{t})=\int\limits_{|a|\leq 1}|a|^{n\left(s+\frac{n-1}{2}\right)} \omega_{\pi_{\nu}}(a)d^\times a=\sum_{j=0}^{\infty}q_\nu^{-jn\left(s+\frac{n-1}{2}\right)}\omega_{\pi_{\nu}}^j(\pomega).
	\end{equation}
	Hence,
	\begin{equation}\label{GJ_over_center2}
		F_{s,\Phi_0}(\underline{t})=\frac{1}{1-\omega_{\pi_{\nu}}(\pomega)q_\nu^{-n\left(s+\frac{n-1}{2}\right)}}=L\left(n\left(s+\frac{n-1}{2}\right),\omega_{\pi_{\nu}}\right).
	\end{equation}
	We now plug \cref{GJ_over_center2} in \cref{GJ_at_spheric_split}.
	\begin{equation}\label{GJ_at_spheric_last}
		\frac{Z_{\mathrm{GJ}}\left(s,c_{v_{\pi_{\nu}}^\circ,\check{v}_{\pi_{\nu}}^{\circ}},\Phi_0\right)}{L\left(n\left(s+\frac{n-1}{2}\right),\omega_{\pi_{\nu}}\right)}=\int\limits_{Z_n(k_\nu)\backslash T^-}\mu(\underline{t})c_{v_{\pi_{\nu}}^\circ,\check{v}_{\pi_{\nu}}^{\circ}}(\underline{t})\left|\det \left( \underline{t}\right)\right|^{s+\frac{n-1}{2}}d\underline{t},
	\end{equation}
	where we divided the $L$-function of $\omega_{\pi_{\nu}}$ from both sides.
	Now by taking $s\mapsto m(s+\frac{1}{2})-\frac{n-1}{2}$ we get that the integral in \cref{GJ_at_spheric_last} equals to $c(I,m(s+\frac{1}{2}))$ in \cref{local_int_unr_split5}. Now we get \cref{unrm_calc_GJ_id} as requested.
\end{proof}
We note that \cref{local_int_unr_split4} shows in particular that $I\left(f_{\omega_{\pi_\nu},s}^\circ,v^\circ_{\pi_{\nu}}\right)$ lies in the space of $\pi_{\nu}$, which is \Cref{local_is_unr}. Also, \Cref{unrm_calc_L_thrm} follows, and, in particular, $I\left(f_{\omega_{\pi_\nu},s}^\circ,v^\circ_{\pi_{\nu}}\right)$ continues to a meromorphic function in $\bbc$, which is a rational function in $q_\nu^{-s}$.

%
%

\subsection{The \enquote{bad} places}\label{bad_places_sec}
\subsubsection{Common statements}
Let $\nu$ be an Archimedean or a ramified non-Archimedean place. Recall that for a section $f_{\omega_{\pi_\nu}}$ in the induced space (given in \cref{local_ind_space_from_omega}):
\begin{equation*}
\rho_{\omega_{\pi_{\nu}},s}=\mathrm{Ind}_{P_{mn-1,1}(k_\nu)}^{\gl[k_\nu]{mn}}\left(1\otimes \omega_{\pi_{\nu}}^{-1}\right) \delta_{P_{mn-1,1}}^{s},
\end{equation*}
we consider the local integrals (given in \cref{local_int}):
\begin{equation*}
I\left(f_{\omega_{\pi_\nu},s},v_{\pi_{\nu}}\right)(h_\nu)=\int \limits _{Z_n(k_\nu)\backslash\gl[k_\nu]{n}}\pi_{\nu}(g_\nu)v_{\pi_{\nu}}  f_{\omega_{\pi_\nu},s}\left(\tilde{\varepsilon}   t\left(h_\nu,g_\nu\right)\right)dg_\nu.
\end{equation*}
This section is dedicated to prove \Cref{local_int_is_pol_prop}.

First, we make some reductions.
We first note that
\begin{equation*}
	f_{\omega_{\pi_\nu},s}\left(\tilde{\varepsilon}   t\left(h_\nu,g_\nu\right)\right)=[\rho_{\omega_{\pi_{\nu}},s}(t(h_\nu,I_n))f_{\omega_{\pi_\nu},s}]\left(\tilde{\varepsilon}   t\left(I_m,g_\nu\right)\right).
\end{equation*}
Hence, by replacing $f_{\omega_{\pi_\nu},s}$ by the smooth section $\rho_{\omega_{\pi_{\nu}},s}(t(h_\nu,I_n))f_{\omega_{\pi_\nu},s}$, we can assume $h_\nu=I_m$ in \cref{local_int}.
Now, we prove
\begin{seclem}\label{local_int_split_prop}
	 $I\left(f_{\omega_{\pi_\nu},s},v_{\pi_{\nu}}\right)(I_m)$ is a finite sum of integrals of the form
	\begin{equation}\label{local_int_split3}
		\alpha\left(\tau,\phi,\phi'\right)\int \limits 	_{Z_n(k_\nu)\backslash\gl[k_\nu]{n}}P_i(q_\nu^{-s}, q_\nu^s)f_{\omega_{\pi_\nu},s}\left(\tilde{\varepsilon}   t\left(I_m,g_\nu\right)\right)\left(\pi_{\nu}(g_\nu)v_{\pi_{\nu}} ,\phi\right)dg_\nu\cdot \phi^{\prime},
	\end{equation}
	where $f_{\omega_{\pi_\nu},s}$ is a standard section, $\tau$ is an irreducible representation of the compact subgroup $t\left(\mathrm{diag}(I_{m-n},K_{n,\nu}),I_n\right)\leq K_{mn,\nu}$, $\phi$ and $\phi'$ are vectors in (an orthonormal basis of) the finite dimensional isotypic subspace $V_\pi(\tau)$, and $\alpha\left(\tau,\phi,\phi'\right)\in \mathbb{C}$.
\end{seclem}
We note that \Cref{local_int_split_prop} also proves that $I\left(f_{\omega_{\pi_\nu},s},v_{\pi_{\nu}}\right)$ is a section in $$\mathrm{Ind}_{P_{m-n,n}(k_\nu)}^{\gl[k_\nu]{m}}\left(1\otimes \tilde{\pi}_\nu\right) \delta_{P_{m-n,n}}^{s}.$$

\begin{proof}
	
	We normalize the measure such that 
	\begin{equation}\label{k_mes_is_1_bad}
		\mu(K_{n,\nu})=\int\limits_{K_{n,\nu}}dx=1.
	\end{equation}
	We begin by writing $f_{\omega_{\pi_\nu},s}$ in terms of standard sections as in \cref{hol_sec_as_standard_local}:
	\begin{equation*}
	f_{\omega_{\pi_\nu},s}=\sum_{i=1}^{N}P_i(q_\nu^{-s}, q_\nu^s)f_{\varphi^{(i)}_{\omega_{\pi_{\nu}}},s},
	\end{equation*}
	where for all $1\leq i \leq N$, $P_i \in\bbc[q_\nu^{-s}, q_\nu^s]$ and $f_{\varphi^{(i)}_{\omega_{\pi_{\nu}}},s}$ is a standard section in $\rho_{\omega_{\pi_{\nu}},s}$.
	Therefore, it suffices to prove the proposition for $I\left(P(q_\nu^{-s}, q_\nu^s)f_{\varphi_{\omega_{\pi_{\nu}}},s},v_{\pi_{\nu}}\right)(I_m)$, where $P(q_\nu^{-s}, q_\nu^s)$ is a holomorphic function (polynomial for non-Archimedean $\nu$) and $f_{\varphi_{\omega_{\pi_{\nu}}},s}$ is a standard section.
	Let $x\in K_{n,\nu}$ and $g_\nu\in \gl[k_\nu]{n}$.
	We now show that 
	\begin{equation}\label{res_section_is_k_fin}
		f_{\varphi_{\omega_{\pi_{\nu}}},s}\left(\tilde{\varepsilon}   t\left(I_m,xg_\nu\right)\right)=\sum_{j=1}^{N}\left<\xi_{\tau_{j},\tilde{\tau}_{j}}\left(x^*\right)\check{\xi}_{\tilde{\tau}_{j}}\right>f^{(j)}_{\varphi_{\omega_{\pi_{\nu}}},s}\left(\tilde{\varepsilon}   t\left(I_m,g_\nu\right)\right),
	\end{equation}
	where for all $1\leq j\leq N$, $\tau_{j}$ is an irreducible representation of the compact subgroup $t\left(\mathrm{diag}(I_{m-n},K_{n,\nu}),I_n\right)\leq K_{mn,\nu}$, $\xi_{\tau_{j}}\in V_{\tau_{j}}$ and $\check{\xi}_{\tilde{\tau}_{j}}\in V_{\tilde{\tau}_{j}}$.
	Indeed, generally, by the $K_{mn,\nu}$-finiteness of $	f_{\varphi_{\omega_{\pi_{\nu}}},s}$, there exist $N\in\mathbb{N}$ and $\{f^{(j)}_{\varphi_{\omega_{\pi_{\nu}}},s}\}_{j=1}^{N}\subseteq \rho_{\omega_{\pi_{\nu}},s}$ such that for all $A\in \gl[k_\nu]{mn}$ and $X\in K_{mn,\nu}$ we have
	\begin{equation}\label{section_is_K_fin}
		f_{\varphi_{\omega_{\pi_{\nu}}},s}\left(AX\right)=\sum_{j=1}^{N} c_{j}(X)f^{(j)}_{\varphi_{\omega_{\pi_{\nu}}},s}(A),
	\end{equation}
	where
	\begin{equation}\label{ind_space_inner_prod}
		c_{j}(X)=\left(\rho_{\omega_{\pi_{\nu}},s}(X)f_{\varphi_{\omega_{\pi_{\nu}}},s},f^{(j)}_{\varphi_{\omega_{\pi_{\nu}}},s}\right)
	\end{equation}
	and $(,)$ is a $K_{mn,\nu}-$invariant inner product of the space of $\rho_{\omega_{\pi_{\nu}},s}$.
    We note that for all $1\leq j \leq N$, $c_{j}(X)$ is independent of $s$, as $f_{\varphi_{\omega_{\pi_{\nu}}},s}$ is a standard section.
	Let $x\in K_{n,\nu}$. Then,
	\begin{equation*}
			f_{\varphi_{\omega_{\pi_{\nu}}},s}\left(\tilde{\varepsilon}   t\left(I_m,xg_\nu\right)\right)=	f_{\varphi_{\omega_{\pi_{\nu}}},s}\left(\tilde{\varepsilon}   t\left(\mathrm{diag}(I_{m-n},x^*)\mathrm{diag}(I_{m-n},(x^*)^{-1}),xg_\nu\right)\right).
	\end{equation*}
	On the other hand, $t(\mathrm{diag}(I_{m-n},x^*),x)\in Q^{\tilde{\varepsilon}}(k_\nu)$, so its conjugation with $\tilde{\varepsilon}$ lies in $P_{mn-1,1}(k_\nu)\cap K_{mn,\nu}$. Since $	f_{\varphi_{\omega_{\pi_{\nu}}},s}$ is $P_{mn-1,1}(k_\nu)\cap K_{mn,\nu}-$left invariant, 
	\begin{equation*}
			f_{\varphi_{\omega_{\pi_{\nu}}},s}\left(\tilde{\varepsilon}   t\left(I_m,xg_\nu\right)\right)=		f_{\varphi_{\omega_{\pi_{\nu}}},s}\left(\tilde{\varepsilon}   t\left(I_m,g_\nu\right)t\left(\mathrm{diag}(I_{m-n},(x^*)^{-1}),I_n\right)\right).
	\end{equation*}
	By \cref{section_is_K_fin},
	\begin{equation}\label{res_section_is_k_fin1}
			f_{\varphi_{\omega_{\pi_{\nu}}},s} \left(  t\left(I_m,xg_\nu\right)\right)=\sum_{j=1}^{N} c^\prime_{j}\left((x^*)^{-1}\right)f^{(j)}_{\varphi_{\omega_{\pi_{\nu}}},s}\left(\tilde{\varepsilon}   t\left(I_m,g_\nu\right)\right),
	\end{equation}
	where
	\begin{equation*}
		c^\prime _{j}\left((x^*)^{-1}\right)=c_{j}\left(t\left(\mathrm{diag}(I_{m-n},(x^*)^{-1}),I_n\right)\right).
	\end{equation*}
	We can assume that $c^\prime _{j}$ are matrix coefficients of an irreducible representation, say $\tau_{j}$, of the compact subgroup $t\left(\mathrm{diag}(I_{m-n},K_{n,\nu}),I_n\right)\leq K_{mn,\nu}$. i.e. by Riesz representation theorem there exist $\xi_{\tau_{j}}\in V_{\tau_{j}}$ and $\check{\xi}_{\tilde{\tau}_{j}}\in V_{\tilde{\tau}_{j}}$ such that
	\begin{equation}\label{pairing_is_mc}
		c^\prime _{j}\left((x^*)^{-1}\right)=\left<\tau_{j}\left(\tilde{w}x^T\tilde{w}\right)\xi_{\tau_{j}},\check{\xi}_{\tilde{\tau}_{j}}\right>=\left<\xi_{\tau_{j}},\tilde{\tau}_{j}\left(x^*\right)\check{\xi}_{\tilde{\tau}_{j}}\right>.
	\end{equation}
	Plugging \cref{pairing_is_mc} to \cref{res_section_is_k_fin1} gives \cref{res_section_is_k_fin}.
	
	Now, we write $I\left(P(q_\nu^{-s}, q_\nu^s)f_{\varphi_{\omega_{\pi_{\nu}}},s},v_{\pi_{\nu}}\right)(I_m)$ as
	\begin{equation}\label{local_int_bed_split}
		\frac{1}{\mu(K_{n,\nu})}\int \limits _{Z_n(k_\nu)\backslash\gl[k_\nu]{n}}\int\limits_{K_{n,\nu}}P(q_\nu^{-s}, q_\nu^s)f_{\omega_{\pi_\nu},s}\left(\tilde{\varepsilon}   t\left(I_m,xg_\nu\right)\right)\pi_{\nu}(xg_\nu)v_{\pi_{\nu}}  dxdg_\nu.
	\end{equation}
	Applying \cref{res_section_is_k_fin} implies that $I\left(P(q_\nu^{-s}, q_\nu^s)f_{\omega_{\pi_\nu},s},v_{\pi_{\nu}}\right)(I_m)$ is a finite sum of integrals of the form
	\begin{equation}\label{local_int_split2}
		\int \limits _{Z_n(k_\nu)\backslash\gl[k_\nu]{n}}P(q_\nu^{-s}, q_\nu^s)f_{\omega_{\pi_\nu},s}\left(\tilde{\varepsilon}   t\left(I_m,g_\nu\right)\right)\int\limits_{K_{n,\nu}} \left<\xi_\tau,\tilde{\tau}\left(x^*\right)\check{\xi}_{\tilde{\tau}}\right>\pi_{\nu}(xg_\nu)v_{\pi_{\nu}} dxdg_\nu.
	\end{equation}
	
	Let $\tau=\tau_\ell$ for some $1\leq \ell \leq N$. Let $P_{\xi_{\tau},\check{\xi}_{\tilde{\tau}}}:V_\pi \to V_\pi$ be the following map
	\begin{equation}\label{proj_to_isotypic_def}
		P_{\xi_{\tau},\check{\xi}_{\tilde{\tau}}}:=\int\limits_{K_{n,\nu}} \left<\xi_\tau,\tilde{\tau}\left(x^*\right)\check{\xi}_{\tilde{\tau}}\right>\pi_{\nu}(x_\nu) dx,
	\end{equation}
	We note that the inner integral over $K_{n,\nu}$ in \cref{local_int_split2} equals $P_{\xi_{\tau},\check{\xi}_{\tilde{\tau}}}\left(\pi_{\nu}(g_\nu)v_{\pi_{\nu}}\right)$.
	Let  $B:=\{\phi_1,\ldots,\phi_r\}$ be an orthonormal basis of of the finite dimensional isotypic subspace $V_\pi(\tau)$. Then,
	\begin{equation}\label{proj_to_isotypic}
		P_{\xi_{\tau},\check{\xi}_{\tilde{\tau}}}\left(\pi_{\nu}(g_\nu)v_{\pi_{\nu}}\right)=\sum_{1\leq i,j\leq r}\left(\pi_{\nu}(g_\nu)v_{\pi_{\nu}} ,\phi_j\right)\int\limits_{K_{n,\nu}} \left<\xi_\tau,\tilde{\tau}\left(x^*\right)\check{\xi}_{\tilde{\tau}}\right>\left(\pi_{\nu}(x^{-1})\phi_{i},\phi_j\right)dx\cdot \phi_i.
	\end{equation}
	Indeed, the map $P_{\xi_{\tau},\check{\xi}_{\tilde{\tau}}}$ is a projection to $V_\pi(\tau)$. Then,
	\begin{equation}\label{proj_to_isotypic1}
		P_{\xi_{\tau},\check{\xi}_{\tilde{\tau}}}\left(\pi_{\nu}(g_\nu)v_{\pi_{\nu}}\right)=\sum_{i=1}^r e_i(g_\nu)\phi_i,
	\end{equation}
	where 
	\begin{equation}\label{proj_to_isotypic_coef}
		e_i(g_\nu)=\left(P_{\xi_{\tau},\check{\xi}_{\tilde{\tau}}}\left(\pi_{\nu}(g_\nu)v_{\pi_{\nu}}\right),\phi_i\right)=\int\limits_{K_{n,\nu}} \left<\xi_\tau,\tilde{\tau}\left(x^*\right)\check{\xi}_{\tilde{\tau}}\right>\left(\pi_{\nu}(xg_\nu)v_{\pi_{\nu}} ,\phi_i\right)dx.
	\end{equation}
	By the fact that $\pi_{\nu}$ is unitary we have $\left(\pi_{\nu}(xg_\nu)v_{\pi_{\nu}} ,\phi_i\right)=\left(\pi_{\nu}(g_\nu)v_{\pi_{\nu}} ,\pi_{\nu}(x^{-1})\phi_i\right)$.
	The vector $\phi_{i}$ is in the finite dimensional space $V_\pi(\tau)$. Hence,
	\begin{equation}\label{isotypic_shift}
		\pi_{\nu}(x^{-1})\phi_{i}=\sum_{j=1}^{r}\left(\pi_{\nu}(x^{-1})\phi_{i},\phi_j\right)\phi_j.
	\end{equation}
	We now plug \cref{isotypic_shift} to \cref{proj_to_isotypic_coef}. Then, we apply the result to \cref{proj_to_isotypic1} and get \cref{proj_to_isotypic}.
	
	The result now followed immediately by denoting
	\begin{equation*}
		\alpha(\tau,\phi,\phi'):=\int\limits_{K_{n,\nu}} \left<\xi_\tau,\tilde{\tau}\left(x^*\right)\check{\xi}_{\tilde{\tau}}\right>\left(\pi_{\nu}(x^{-1})\phi',\phi\right)dx,
	\end{equation*}
	which absolutely converges as it is an integral of continuous function on a compact subgroup.
	
\end{proof}

We state two corollaries obtained from this proof, 
\begin{cor}\label{k_inv_in_bad_places}
	Let $g_\nu \in \gl[k_\nu]{n}$, $x\in K_n$, and $f_{\omega_{\pi_\nu},s}$ a standard section. Then,  $f_{\omega_{\pi_\nu},s}\left(\tilde{\varepsilon}   t\left(I_m,g_\nu x\right)\right)$ equals to a finite sum of elements of the form $c(x)f_{\omega_{\pi_\nu},s}\left(\tilde{\varepsilon}   t\left(I_m,g_\nu \right)\right)$, where $c(x)$ is independent of $s$.
\end{cor}
\begin{proof}
	This is \cref{section_is_K_fin}.
\end{proof}
\begin{cor}\label{section_are_bi_inv}
	Let $g_\nu=A \underline{t} B$, be the Cartan decomposition (modulo the center) of $g_\nu$ as in \cref{cartan_dec}, and $f_{\omega_{\pi_\nu},s}$ a standard section. Then, $f_{\omega_{\pi_\nu},s}\left(\tilde{\varepsilon}   t\left(I_m,g_\nu \right)\right)$ equals to a finite sum of elements of the form $c(A,B)f_{\omega_{\pi_\nu},s}\left(\tilde{\varepsilon}   t\left(I_m, \underline{t}\right)\right)$, where $c(A,B)$ is independent of $s$.
\end{cor}
\begin{proof}
	This follows immediately from \cref{res_section_is_k_fin} and \cref{k_inv_in_bad_places}.
\end{proof}

In the next subsections we split to the cases of non-Archimedean and Archimedean places.

\subsubsection{Non-Archimedean places}\label{nonarch_places_sec}
Let $\nu<\infty$.
We now denote $\underline{t}^\Delta:=t\left(I_m, \underline{t} \right)$. Recall that $\tilde{\varepsilon}=\tilde{w}\tilde{u}$ where we write $\tilde{w}:=w_{(m-n+1)n}$ and $\tilde{u}:=u_{(m-n+1)n}(\underline{b_{n-1}})$. We also denote $\tilde{u}(\underline{t}):=u_{(m-n+1)n}(\underline{b_{n-1}}(\underline{t}))$, where 
\begin{equation*}
	\underline{b_{n-1}}(\underline{t})=\left(t_{n-1}\cdot e_{n-1}^T,\ldots,t_1\cdot e_{1}^T\right).
\end{equation*}

\begin{prop}\label{sec_val_on_non_arch_det}
	Let $\underline{t}\in Z_n(k_\nu)\backslash T^-$. Then, for a standard section $f_{\omega_{\pi_\nu},s}$ we have
	\begin{equation*}
		f_{\omega_{\pi_\nu},s}\left(\tilde{\varepsilon}   \underline{t}^\Delta\right)= \sum_{j=1}^N	c_{j}(\tilde{w}\tilde{u}(\underline{t}))\left|\det \underline{t}\right|^{ms+\frac{m}{2}},
	\end{equation*}
	where for $1\leq j\leq N$, $c_{j}$ is the inner product of the space of $\rho_{\omega_{\pi_{\nu}},s}$ given in \cref{ind_space_inner_prod}.
\end{prop}

\begin{proof}

We have
\begin{equation*}
	\tilde{\varepsilon} \underline{t}^\Delta=\left(\tilde{w}\underline{t}^\Delta\tilde{w}^{-1}\right)\tilde{w}\left(\left(\underline{t}^\Delta\right)^{-1}\tilde{u}\underline{t}^\Delta\right)=\left(\tilde{w}\underline{t}^\Delta\tilde{w}^{-1}\right)\tilde{w}\tilde{u}(\underline{t}).
\end{equation*}
All the elements above the diagonal of the unipotent matrix $u_{(m-n+1)n}(\underline{b_{n-1}}(\underline{t}))$ are bounded in absolute value by $1$. Therefore, $\tilde{w}\tilde{u}(\underline{t})\in K_{mn,\nu}$. \Cref{k_inv_in_bad_places} implies that $f_{\omega_{\pi_\nu},s}\left(\tilde{\varepsilon}  \underline{t}^\Delta\right)$ equals to a finite sum of elements of the form $c_{j}(\tilde{w}\tilde{u}(\underline{t}))f_{\omega_{\pi_\nu},s}\left(\tilde{w}\underline{t}^\Delta\tilde{w}^{-1}\right)$.
Now, by the same arguments as in the proof of \Cref{sec_val_on_tor} starting with \cref{conj_torus_by_perm}, we find that $f_{\omega_{\pi_\nu},s} \left(\tilde{w}\underline{t}^\Delta\tilde{w}^{-1}\right)=\prod_{i=1}^{n-1}\left|t_i\right|^{ms+\frac{m}{2}}$ and the proposition follows.
\end{proof}

\begin{cor}\label{local_int_det_prop}
	Assume that $\Re(s)>>0$. In the notation of \Cref{local_int_split_prop,sec_val_on_non_arch_det}, $I\left(f_{\omega_{\pi_\nu},s},v_{\pi_{\nu}}\right)(I_m)$ is a finite sum of integrals of the form
	\begin{equation}\label{local_int_det_eq}
		\alpha\left(\tau,\phi,\phi'\right)\beta\int \limits 	_{Z_n(k_\nu)\backslash T^-}P(q_\nu^{-s}, q_\nu^s)\mu(\underline{t})c(X(\tilde{\varepsilon},\underline{t}))\left|\det \underline{t}\right|^{ms+\frac{m}{2}}\left(\pi_{\nu}(\underline{t})v_{\pi_{\nu}} ,\phi\right)d\underline{t}\cdot  \phi^{\prime},
	\end{equation}
	 where $P(q_\nu^{-s}, q_\nu^s)$ is a holomorphic function (polynomial for non-Archimedean $\nu$), $\mu(\underline{t})=\mu(K_{n,\nu}\underline{t}K_{n,\nu})$, and $c$ is the inner product of the space of $\rho_{\omega_{\pi_{\nu}},s}$ given in \cref{ind_space_inner_prod}.
\end{cor}

\begin{proof}
	By \Cref{local_int_split_prop}, $I\left(f_{\omega_{\pi_\nu},s},v_{\pi_{\nu}}\right)(I_m)$ is a finite sum of integrals of the form
	\begin{equation}\label{local_int_split_prop2}
		\alpha\left(\tau,\phi,\phi'\right)\int \limits 	_{Z_n(k_\nu)\backslash\gl[k_\nu]{n}}P(q_\nu^{-s}, q_\nu^s)f_{\omega_{\pi_\nu},s}\left(\tilde{\varepsilon}   t\left(I_m,g_\nu\right)\right)\left(\pi_{\nu}(g_\nu)v_{\pi_{\nu}} ,\phi\right)dg_\nu \cdot \phi^{\prime}.
	\end{equation}
	By \Cref{section_are_bi_inv}, $f_{\omega_{\pi_\nu},s}\left(\tilde{\varepsilon}   t\left(I_m,g_\nu\right)\right)$ equals to a finite sum of $c(A,B)f_{\omega_{\pi_\nu},s}\left(\tilde{\varepsilon}   t\left(I_m, \underline{t}\right)\right)$, where $c(A,B)$ is independent of $s$.
 	Since $v_{\pi_\nu}$ and $\phi$ are $K_{n,\nu}$-finite, we may replace $\left(\pi_{\nu}(g_\nu)v_{\pi_{\nu}} ,\phi\right)$ by $\left(\pi_{\nu}(\underline{t})v_{\pi_{\nu}} ,\phi\right)$.
 	Hence, by \Cref{kak_integral} we get that $I\left(f_{\omega_{\pi_\nu},s},v_{\pi_{\nu}}\right)(I_m)$ is a finite sum of terms of the form
 	\begin{equation}\label{local_int_split_prop3}
 		\alpha\left(\tau,\phi,\phi'\right)\int \limits 	_{K_{n,\nu}^2}c(A,B)dAdB\int \limits 	_{Z_n(k_\nu)\backslash T^-}P(q_\nu^{-s}, q_\nu^s)\mu(\underline{t})f_{\omega_{\pi_\nu},s}\left(\tilde{\varepsilon}   t\left(I_m,\underline{t}\right)\right)\left(\pi_{\nu}(\underline{t})v_{\pi_{\nu}} ,\phi\right)d\underline{t} \cdot \phi^{\prime},
 	\end{equation}
 	where $\mu(\underline{t})=\mu(K_{n,\nu}\underline{t}K_{n,\nu})$.
 	The integral $\beta:=\int \limits 	_{K_{n,\nu}^2}c(A,B)dAdB$ absolutely converges.
 	Finally, \Cref{sec_val_on_non_arch_det} implies that \cref{local_int_split_prop3} equals to a finite sum of terms of the form
	\begin{equation*}
		\alpha\left(\tau,\phi,\phi'\right)\beta\int \limits _{Z_n(k_\nu)\backslash T^-}P(q_\nu^{-s}, q_\nu^s)\mu(\underline{t})c(X(\tilde{\varepsilon},\underline{t}))\left|\det \underline{t}\right|^{ms+\frac{m}{2}}\left(\pi_{\nu}(\underline{t})v_{\pi_{\nu}} ,\phi\right)d\underline{t} \cdot \phi^{\prime},
	\end{equation*}
	where $c$ is the inner product of the space of $\rho_{\omega_{\pi_{\nu}},s}$ given in \cref{ind_space_inner_prod}.
\end{proof}

Next, we make use of the asymptotic behavior of matrix coefficients due to \cite{casselman1995introduction} (quoted from \cite{hazan2022asym} for convenience).
For each $\Theta\subseteq \Delta$ and $0<\varepsilon\leq 1$, we define
\begin{equation*}
T_{\Theta}^{-}\left(\varepsilon\right)=\left\{ a\in T\left| \begin{array}{l}
\left|\alpha\left(a\right)\right| \leq\varepsilon\ \forall\alpha\in\Delta\backslash \Theta \\ \varepsilon < \left|\alpha\left(a\right)\right|\leq 1\ \forall\alpha\in \Theta
\end{array} \right.\right\}.
\end{equation*}
This is a subset of $T^-$. For  each $\Theta\subseteq \Delta$ we denote by $P_\Theta=M_\Theta N_\Theta$  the standard parabolic subgroup corresponding to $\Theta$.
\begin{Theorem*}[Theorem 1 in \cite{hazan2022asym}]
	Let $v\in V$ and $\tilde{v}\in\tilde{V}$. There exist $\varepsilon>0$ and finite sets of vectors, that depend on $\{\pi_\nu, v, \tilde{v}\}$, $\underline{p'}=\left(p'_{1},\ldots,p'_{r}\right)\in \bbr^r,\ \underline{p}=\left(p_{1},\ldots,p_{r}\right)\in \bbz^r_{\ge 0}$, and $\underline{\chi}=\left(\chi_{1},\ldots,\chi_{r}\right)$ where for all $1\leq i\leq r$, $\chi_i:k^\times \to \bbc^\times$ are unitary characters, such that for all $a\in T^-$, one has
	\begin{equation}\label{asym_exp_p_adic}
		\left\langle \pi_\nu(a)v,\tilde{v}\right\rangle = \sum_{\Theta\subseteq \Delta,\underline{p'},\underline{p},\underline{\chi}}\chi_{A_\Theta^-\left(\varepsilon\right)}(a)\alpha_{\underline{p'},\underline{p},\underline{\chi}}\prod_{i=1}^{r_\Theta}\chi_i(a_i)\left|a_{i}\right|^{p'_{i}}\log_q^{p_{i}}\left|a_{i}\right|,
	\end{equation}
	where  $\chi_{T_\Theta^-\left(\varepsilon\right)}(a)$ is the indicator function of $T_\Theta^-\left(\varepsilon\right)$, $r_\Theta$ is such that $T_{M_\Theta}\cong (k^\times)^{r_\Theta}$ by the map $a\mapsto\left(a_1,\ldots, a_{r_\Theta}\right)$, and $\alpha_{\underline{p'},\underline{p},\underline{\chi}}\in \bbc$ are such that $\alpha_{\underline{p'},\underline{p},\underline{\chi}}=0$ for all but finitely many  $\underline{p'},\underline{p},\underline{\chi}$.
\end{Theorem*}

Let $g=(g_{i,j})_{1\leq i,j\leq n}\in\gl[k_\nu]{n}$. Define $||g||=\max_{1\leq i,j\leq n}\{|g_{i,j}|,|g^{-1}_{i,j}|\}$. Then, \cref{asym_exp_p_adic} gives in particular,
\begin{cor}
	There exist constants $\lambda,m_a\in\mathbb{R}_+$ such that
	\begin{equation*}
		\left|\left\langle \pi(a)v,\tilde{v}\right\rangle \right| \leq \lambda ||a|| ^{m_a}.
	\end{equation*}
\end{cor}

\begin{prop}
	The integral $I\left(f_{\omega_{\pi_\nu},s},v_{\pi_{\nu}}\right)(I_m)$ absolutely converges in $\Re(s)>>0$.
\end{prop}
\begin{proof}
	For any $0<\varepsilon \leq 1$, $T^-$ is the disjoint union of $T_{\Theta}^{-}\left(\varepsilon\right)$ as $\Theta$ ranges over all subsets of $\Delta$. Thus, together with \Cref{local_int_det_prop}, $I\left(f_{\omega_{\pi_\nu},s},v_{\pi_{\nu}}\right)(I_m)$ is a finite sum of integrals of the form
	\begin{equation*}
	\alpha\left(\tau,\phi,\phi'\right)\beta\int \limits 	_{Z_n(k_\nu)\backslash T_\Theta^-(\varepsilon)}P(q_\nu^{-s}, q_\nu^s)\mu(\underline{t})c(X(\tilde{\varepsilon},\underline{t}))\left|\det \underline{t}\right|^{ms+\frac{m}{2}}\left(\pi_{\nu}(\underline{t})v_{\pi_{\nu}} ,\phi\right)d\underline{t}\cdot  \phi^{\prime}.
	\end{equation*}
	where, $P(q_\nu^{-s}, q_\nu^s)$ is polynomial,  $c(X(\tilde{\varepsilon},\underline{t}))$ does not depend on $s$, $\mu(\underline{t})= \frac{\mu_{M_a}}{\mu_G} \delta^{-1}_B(a)$ (by \Cref{jacobian_of_kak_int_non_arch}).
	We have $||\underline{t}||=|t_1|^{-1}$. Therefore, each of the integrals above are bounded in absolute value by
	\begin{equation*}
	\begin{split}
		\int \limits 	_{|t_1|\leq |t_2|\leq \ldots |t_{n-1}|\leq 1}&|t_1|^{1-n-m_a+ms+m/2}|t_2|^{3-n+ms+m/2}\\&\cdots |t_{n-1}|^{n-3+ms+m/2}d^\times t_1 \cdots d^\times t_{n-1} .
	\end{split}
	\end{equation*}
	By separate variables we find that the last integral equals
	\begin{equation*}
		\int \limits 	_{|t_1|\leq 1}|t_1|^{1-n-m_a+ms+m/2}d^\times t_1\prod_{j=2}^{n-1}	\int \limits 	_{|t_j|\leq 1}|t_j|^{j+1-n+ms+m/2}d^\times t_j.
	\end{equation*}
	Hence, it is absolutely convergent iff $\Re(1-n-m_a+ms+m/2)\geq 0$.
\end{proof}

Moreover, the integral $I\left(f_{\omega_{\pi_\nu},s},v_{\pi_{\nu}}\right)(I_m)$  has meromorphic continuation. We show it by first applying \cref{asym_exp_p_adic} to \Cref{local_int_det_prop}.

\begin{cor}\label{local_int_det_prop2}
	Assume that $\Re(s)>>0$. In the notation of \Cref{local_int_split_prop,sec_val_on_non_arch_det}, $I\left(f_{\omega_{\pi_\nu},s},v_{\pi_{\nu}}\right)(I_m)$ is a finite sum of integrals of the form
	\begin{equation}\label{local_int_det_eq2}
	\begin{split}
		\alpha\left(\tau,\phi,\phi'\right)\beta\int \limits 	_{Z_n(k_\nu)\backslash T_\Theta^-(\varepsilon)}&P(q_\nu^{-s}, q_\nu^s)\mu(\underline{t})c(X(\tilde{\varepsilon},\underline{t}))\left|\det \underline{t}\right|^{ms+\frac{m}{2}}\\&\cdot\alpha_{\underline{p},\underline{p'},\underline{\chi}}\prod_{i=1}^{r_\Theta}\chi_i(t_i)\left|t_{i}\right|^{p_{i}}\log^{p'_{i}}\left|t_{i}\right|d\underline{t}\cdot  \phi^{\prime},
	\end{split}
	\end{equation}
	where $P(q_\nu^{-s}, q_\nu^s)$ is a holomorphic function (polynomial for non-Archimedean $\nu$), $\mu(\underline{t})=\mu(K_{n,\nu}\underline{t}K_{n,\nu})$, and $c$ is the inner product of the space of $\rho_{\omega_{\pi_{\nu}},s}$ given in \cref{ind_space_inner_prod}.
\end{cor}

\begin{prop}
	The integral $I\left(f_{\omega_{\pi_\nu},s},v_{\pi_{\nu}}\right)(I_m)$ has meromorphic continuation for  $\Re(s)>>0$.
\end{prop}
\begin{proof}
	By \Cref{local_int_det_prop}, $I\left(f_{\omega_{\pi_\nu},s},v_{\pi_{\nu}}\right)(I_m)$ is a finite sum of integrals of the form
	\begin{equation*}
	\begin{split}
		\alpha\left(\tau,\phi,\phi'\right)\beta\int \limits 	_{Z_n(k_\nu)\backslash T_\Theta^-(\varepsilon)}&P(q_\nu^{-s}, q_\nu^s)\mu(\underline{t})c(X(\tilde{\varepsilon},\underline{t}))\left|\det \underline{t}\right|^{ms+\frac{m}{2}}\\&\cdot\alpha_{\underline{p},\underline{p'},\underline{\chi}}\prod_{i=1}^{r_\Theta}\chi_i(t_i)\left|t_{i}\right|^{p_{i}}\log^{p'_{i}}\left|t_{i}\right|d\underline{t}\cdot  \phi^{\prime}.
	\end{split}
	\end{equation*}
	where, $P(q_\nu^{-s}, q_\nu^s)$ is polynomial,  $c(X(\tilde{\varepsilon},\underline{t}))$ does not depend on $s$, and by \Cref{jacobian_of_kak_int_non_arch},
	\begin{equation*}
	\mu(\underline{t})=q_\nu ^{\sum_{i=1}^n(n-2i+1)r_i} \frac{\varphi_n(q^{-1})}{(1-q^{-1})^n} \prod_{j=1}^\ell \frac{(1-q^{-1})^{n_j}}{\varphi_{n_j}(q^{-1})}.
	\end{equation*}
	
	We have $||\underline{t}||=|t_1|^{-1}$. Therefore, each of the integrals above can be brought to the form
	\begin{equation*}
	\begin{split}
		\int \limits 	_{|t_1|\leq |t_2|\leq \ldots |t_{n-1}|\leq 1}&|t_1|^{1-n+ms+m/2}|t_2|^{3-n+ms+m/2}\cdots |t_{n-1}|^{n-3+ms+m/2}\\&\prod_{i=1}^{r_\Theta}\chi_i(t_i)\left|t_{i}\right|^{p_{i}}\log^{p'_{i}}\left|t_{i}\right|d^\times t_1 \cdots d^\times t_{n-1} .
	\end{split}
	\end{equation*}
	By separating variables we find that the last integral equals
	\begin{equation}\label{merom_cont_int_sep}
		\prod_{j=1}^{n-1}	\int \limits 	_{|t_j|\leq 1}\chi_j(t_j)|t_j|^{\alpha_j+ms}\log^{p'_{j}}\left|t_{j}\right|d^\times t_j.
	\end{equation}
	We can assume $\chi_i(t_i)=0$ for all $i$, otherwise the integral will be zero.
	Therefore, \cref{merom_cont_int_sep} equals
	\begin{equation*}
			\prod_{j=1}^{n-1}\sum_{i=0}^\infty iq_\nu^{-i(\alpha_j+ms)}.
	\end{equation*}
	The last sum is a derivative of a rational function, and the meromorphic continuation follows.
\end{proof}

\subsubsection{Archimedean places}
In this section we briefly provide Archimedean analogues of the results in the \Cref{nonarch_places_sec}. Let $\nu =\infty $.

\begin{prop}\label{sec_val_on_arch_det}
	Let $\underline{t}\in Z_n(k_\nu)\backslash T^-$. Then, for a standard section $f_{\omega_{\pi_\nu},s}$ we have 
	\begin{equation*}
	f_{\omega_{\pi_\nu},s}\left(\tilde{\varepsilon}   \underline{t}^\Delta\right)= c(\tilde{\omega},\underline{t})\left|\det \underline{t}\right|^{ms+\frac{m}{2}}\omega_{\pi_{\nu}}\left(1+\sum_{i=1}^{n-1}t_j^2\right)\left|1+\sum_{i=1}^{n-1}t_j^2\right|^{-mn(s+\frac{1}{2})},
	\end{equation*}
	where $c(\tilde{\omega},\underline{t})$ is the inner product of the space of $\rho_{\omega_{\pi_{\nu}},s}$ given in \cref{ind_space_inner_prod} evaluated on a matrix in $K_{mn}$ that depends on $\tilde{\omega}$ and $\underline{t}$.
\end{prop}

\begin{proof}
	In any case $\nu=\mathbb{R}$ or $\nu=\mathbb{C}$, for all $1\leq j\leq n-1$, we can write $t_j=r_j e^{i\theta_j}$, where $0<r_1\leq r_2\leq ...\leq r_{n-1} \leq 1$ and $-\pi\leq \theta_j<\pi$. We write
	\begin{equation*}
	\underline{t}=\mathrm{diag}(r_1,\ldots,r_{n-1},1)\mathrm{diag}(e^{i\theta_1},\ldots,e^{i\theta_{n-1}},1).
	\end{equation*}
	By \Cref{k_inv_in_bad_places} we have
	\begin{equation*}
	f_{\omega_{\pi_\nu},s}\left(\tilde{\varepsilon} \underline{t}^\Delta\right)=c(\underline{\theta})f_{\omega_{\pi_\nu},s}\left(\tilde{\varepsilon} \underline{r}^\Delta \right),
	\end{equation*}
	where $\underline{\theta}=\mathrm{diag}(e^{i\theta_1},\ldots,e^{i\theta_{n-1}},1)$ and $c(\underline{\theta})$ is a holomorphic function in $q_\nu ^{\pm s}$.
	Therefore, we assume that  $0<t_1\leq t_2\leq \ldots\leq t_{n-1} \leq 1$.
	Following the proof of \Cref{sec_val_on_non_arch_det} we get that all the arguments hold except that this time $\tilde{w}\tilde{u}(\underline{t})\notin K_{mn,\nu}$. Hence, 
	\begin{equation}\label{sec_val_arc_place}
	f_{\omega_{\pi_\nu},s}\left(\tilde{\varepsilon}   \underline{t}^\Delta\right)= \left|\det \underline{t}\right|^{ms+\frac{m}{2}}f_{\omega_{\pi_\nu},s}\left(\tilde{w} \underline{t}^\Delta\tilde{w}^{-1}\tilde{w}\right).
	\end{equation}
	We apply \Cref{k_inv_in_bad_places} again
	\begin{equation}\label{sec_val_arc_place_wo_det}
	f_{\omega_{\pi_\nu},s}\left(\tilde{w} \underline{t}^\Delta\right)=c(\tilde{w})f_{\omega_{\pi_\nu},s}\left(\tilde{w} \underline{t}^\Delta \tilde{w}^{-1}\right),
	\end{equation}
	where $c(\tilde{w})$ is a holomorphic function in $q_\nu ^{\pm s}$. Now
	\begin{equation*}
	\tilde{w} \underline{t}^\Delta \tilde{w}^{-1}=\begin{pmatrix}
	I_{mn-1}&0\\x(\underline{t})&1
	\end{pmatrix},
	\end{equation*} 
	where the last row equals
	\begin{equation}\label{last_row_befor_iwa_dec}
	(x(\underline{t}),1)=(0_{(m-n+1)n},t_{n-1}e_{n-2},\ldots,t_3e_2,t_2e_1+t_1e_n,e_n).
	\end{equation} 
	We would like to find the Iwasawa decomposition  $\tilde{w} \underline{t}^\Delta \tilde{w}^{-1}=x_Px_K$, such that $x_P\in P_{mn-1,1}(k_\nu)$ and $x_K\in K_{mn,\nu}$. This implies 
	\begin{equation}\label{section_val_on_iwa_arch}
	f_{\omega_{\pi_\nu},s}\left(\tilde{w} \underline{t}^\Delta \tilde{w}^{-1}\right)=c(x_K)\left(1\otimes \omega_{\pi_{\nu}}\right)(x_P)\delta_{P_{mn-1,1}}^{s+\frac{1}{2}}(x_P).
	\end{equation}
	Most of the rows and columns of $\tilde{w} \underline{t}^\Delta \tilde{w}^{-1}$ are already the standard orthonormal basis of $k_\nu^{mn}$. Therefore, it is sufficient to find the Iwasawa decomposition
	\begin{equation}\label{iwasawa_dec_arch}
	\begin{pmatrix}
	I_{n-1}&0\\y(\underline{t})&1
	\end{pmatrix}=y_Py_K,
	\end{equation} 
	where the last row equals
	\begin{equation}\label{last_row_before_g_sc}
	(y(\underline{t}),1)=(t_{n-1},\ldots,t_1,1).
	\end{equation}
	We denote
	\begin{equation*}
	y_P^{-1}=\begin{pmatrix}
	X_{n,n-1}&\ldots&X_{n,0}\\\vdots&&\vdots\\X_{1,n-1}&&X_{1,0}
	\end{pmatrix}, \qquad \begin{pmatrix}
	I_{n-1}&0\\y(\underline{t})&1
	\end{pmatrix}=\begin{pmatrix}
	-u_1-\\\vdots\\-u_n-
	\end{pmatrix}, \qquad y_K=\begin{pmatrix}
	-v_n-\\\vdots\\-v_1-
	\end{pmatrix}.
	\end{equation*}
	In this notation \cref{iwasawa_dec_arch} can be written as the following system of equation. For all $1\leq i \leq n$,
	\begin{equation*}
	v_i=\sum_{r=0}^{n-1}X_{i,r}u_{n-r}.
	\end{equation*}  
	The Gram–Schmidt process gives for all $1\leq i \leq n$
	\begin{equation*}
	v_i^\prime=\sum_{r=0}^{n-1}X_{i,r}^\prime u_{n-r}.
	\end{equation*}
	This provides $y_K$ by $v_i=\frac{v_i^\prime}{\left|\left|v_i^\prime\right|\right|}$, and $y_P^{-1}$ by  $X_{i,r}=\frac{X_{i,r}^\prime}{\left|\left|v_i^\prime\right|\right|}$.
	We set $v_1^\prime=u_n$. For all $1\leq i \leq n$ we denote
	\begin{equation*}
	\varphi_i=1+\sum_{j=i}^{n-1} t_j^2.
	\end{equation*}
	We note that $\varphi_1=\left|\left|u_n\right|\right|^2=\left|\left|v_1^\prime\right|\right|^2$, $\varphi_n=1$, and
	\begin{equation}\label{norm_recursive_rel}
	\varphi_i+ t_{i-1}^2=\varphi_{i-1}.
	\end{equation}
	We show by induction that for $2\leq i\leq n$ we get
	\begin{align}
	v_i^\prime&=u_{n-i+1}+\frac{t_{i-1}}{\varphi_{i-1}}\left(\sum_{r=1}^{i-2}t_ru_{n-r}-u_n\right),\label{recursive_v_i} \\ 
	\left|\left|v_i^\prime\right|\right|^2&=\frac{\varphi_i}{\varphi_{i-1}}.\label{recursive_norm_v_i}
	\end{align}
	Indeed, by applying the Gram-Schmidt process for $i=2$ we have
	\begin{equation*}
	v_2^\prime=u_{n-1}-\frac{\left<u_{n-1},v_1^\prime\right>}{\left|\left|v_1^\prime\right|\right|^2}v_1^\prime=u_{n-1}-\frac{t_1}{\varphi_1}u_n.
	\end{equation*}
	This also implies,
	\begin{equation*}
	v_2^\prime=e_{n-1}-\frac{t_1}{\varphi_1}\left(\sum_{r=1}^{n-1}t_re_{n-r}+e_n\right)=\left(1-\frac{t_1^2}{\varphi_1}\right)e_{n-1}-\frac{t_1}{\varphi_1}\left(\sum_{r=2}^{n-1}t_re_{n-r}+e_n\right).
	\end{equation*}
	So,
	\begin{equation*}
	\left|\left|v_2^\prime\right|\right|^2=\left(1-\frac{t_1^2}{\varphi_1}\right)^2+\frac{t_1^2}{\varphi_1^2}\left(\sum_{r=2}^{n-1}t_r^2+1\right)=\frac{\varphi_2^2}{\varphi_1^2}+\frac{t_1^2}{\varphi_1^2}\varphi_2=\frac{\varphi_2}{\varphi_1}\left(\frac{\varphi_2+t_1^2}{\varphi_1}\right)=\frac{\varphi_2}{\varphi_1}.
	\end{equation*}
	Let $2<i\leq n$ and assume the induction hypothesis \cref{recursive_v_i,recursive_norm_v_i} is true for all $2\leq j <i$. Gram-Schmidt process gives
	\begin{equation}\label{gram_schmidt_v_i}
	v_i^\prime=u_{n-i+1}-\frac{\left<u_{n-i+1},v_1^\prime\right>}{\left|\left|v_1^\prime\right|\right|^2}v_1^\prime-\sum_{\ell=2}^{i-1}\frac{\left<u_{n-i+1},v_\ell^\prime\right>}{\left|\left|v_\ell^\prime\right|\right|^2}v_\ell^\prime.
	\end{equation}
	Now,
	\begin{equation}\label{first_summand_g_sc}
	\frac{\left<u_{n-i+1},v_1^\prime\right>}{\left|\left|v_1^\prime\right|\right|^2}v_1^\prime=\frac{t_{i-1}}{\varphi_1}u_n,
	\end{equation}
	and by the induction hypothesis we have for all $1< \ell < i$ 
	\begin{equation*}
	\left<u_{n-i+1},v_\ell^\prime\right>=\left<u_{n-i+1},u_{n-\ell+1}\right>+\frac{t_{\ell-1}}{\varphi_{\ell-1}}\left(\sum_{r=1}^{\ell-2}t_r\left<u_{n-i+1},u_{n-r}\right>-\left<u_{n-i+1},u_n\right>\right).
	\end{equation*}
	So,
	\begin{equation}\label{in_prod_g_sc}
	\left<u_{n-i+1},v_\ell^\prime\right>=\frac{t_{\ell-1}t_{i-1}}{\varphi_{\ell-1}}.
	\end{equation}
	By the induction hypothesis, \cref{first_summand_g_sc}, and \cref{in_prod_g_sc}, we can rewrite \cref{gram_schmidt_v_i}
	\begin{equation*}
	v_i^\prime=u_{n-i+1}-\frac{t_{i-1}}{\varphi_1}u_n-\sum_{\ell=2}^{i-1}\frac{t_{\ell-1}t_{i-1}}{\varphi_{\ell}}u_{n-\ell+1}+\sum_{\ell=2}^{i-1}\frac{t_{\ell-1}^2t_{i-1}}{\varphi_{\ell}\varphi_{\ell-1}}\left(\sum_{r=1}^{\ell-2}t_ru_{n-r}-u_n\right).
	\end{equation*}
	By rearranging and changing the order of summation we have
	\begin{equation}\label{linear_comb_v_i}
	v_i^\prime=u_{n-i+1}-\left(\frac{t_{i-1}}{\varphi_1}+\sum_{\ell=2}^{i-1}\frac{t_{\ell-1}^2t_{i-1}}{\varphi_{\ell}\varphi_{\ell-1}}\right)u_n-\sum_{r=1}^{i-2}\frac{t_{r}t_{i-1}}{\varphi_{r+1}}u_{n-r}+\sum_{r=1}^{i-3}t_ru_{n-r}\sum_{\ell=r+2}^{i-1}\frac{t_{\ell-1}^2t_{i-1}}{\varphi_{\ell}\varphi_{\ell-1}}.
	\end{equation}
	
	We use \cref{linear_comb_v_i} to find the coefficients $X_{i,r}^\prime$ for all $0\leq r\leq n-1$. It is immediate to see that
	\begin{equation*}
	\begin{cases}
	X_{i,r}^\prime=0,&i\leq r\leq n-1\\
	X_{i,i-1}^\prime=1,&\\
	X_{i,i-2}^\prime=\frac{t_{i-1}t_{i-2}}{\varphi_{i-1}},&
	\end{cases}
	\end{equation*} 
	For $1\leq r \leq i-3$,
	\begin{equation}\label{coef_of_linear_comb_v_i}
	X_{i,r}^\prime = t_rt_{i-1}\left(\frac{1}{\varphi_{r+1}}+\sum_{\ell=r+2}^{i-1}\frac{t_{\ell-1}^2}{\varphi_{\ell}\varphi_{\ell-1}}\right).
	\end{equation}
	Generally, for $1\leq a\leq n-2$,
	\begin{equation}\label{varphi_a_id}
	\frac{1}{\varphi_{a}}+\frac{t_{a}^2}{\varphi_{a+1}\varphi_{a}}=\frac{1}{\varphi_{a}}\left(1+\frac{t_{a}^2}{\varphi_{a+1}}\right)=\frac{1}{\varphi_{a}}\frac{\varphi_{a}}{\varphi_{a+1}}=\frac{1}{\varphi_{a+1}}.
	\end{equation}
	Thus, the sum in the parenthesis in \cref{coef_of_linear_comb_v_i} equals $\varphi_{i-1}^{-1}$ and $X_{i,r}^\prime = t_rt_{i-1}\varphi_{i-1}^{-1}$, for all $1\leq r \leq i-3$.
	This applies to $r=0$ as well, 
	\begin{equation}\label{coef_of_linear_comb_v_i_0}
	X_{i,0}^\prime = t_{i-1}\left(\frac{1}{\varphi_{1}}+\sum_{\ell=2}^{i-1}\frac{t_{\ell-1}^2}{\varphi_{\ell}\varphi_{\ell-1}}\right)=\frac{t_{i-1}}{\varphi_{i-1}}.
	\end{equation}
	This implies that \cref{recursive_v_i} holds true for all $2\leq i\leq n$.
	We now plug $u_j=e_j$ for all $1\leq j \leq n-1$ and \cref{last_row_before_g_sc} in \cref{recursive_v_i}:
	\begin{equation*}
	v_i^\prime=e_{n-i+1}+\frac{t_{i-1}}{\varphi_{i-1}}\left(\sum_{r=1}^{i-2}t_re_{n-r}-\sum_{r=1}^{n-1}t_re_{n-r}+e_n\right),
	\end{equation*}
	and rearrange
	\begin{equation*}
	v_i^\prime=\left(1-\frac{t_{i-1}^2}{\varphi_{i-1}}\right)e_{n-i+1}-\frac{t_{i-1}}{\varphi_{i-1}}\sum_{r=i}^{n-1}t_re_{n-r}+\frac{t_{i-1}}{\varphi_{i-1}}e_n.
	\end{equation*}
	Therefore,
	\begin{equation*}
	\left|\left|v_i^\prime\right|\right|^2=\left(1-\frac{t_{i-1}^2}{\varphi_{i-1}}\right)^2+\frac{t_{i-1}^2}{\varphi_{i-1}^2}\sum_{r=i}^{n-1}t_r^2+\frac{t_{i-1}^2}{\varphi_{i-1}^2}=\frac{\varphi_{i}^2}{\varphi_{i-1}^2}+\frac{t_{i-1}^2\varphi_{i}}{\varphi_{i-1}^2}=\frac{\varphi_{i}}{\varphi_{i-1}},
	\end{equation*}
	and
	\begin{equation*}
	v_i=e_{n-i+1}-\frac{t_{i-1}}{\varphi_{i}}\sum_{r=i}^{n-1}t_re_{n-r}+\frac{t_{i-1}}{\varphi_{i}}e_n.
	\end{equation*}
	Thereby, $y_P^{-1}$ is an upper triangular matrix with diagonal 
	\begin{equation*}
	\mathrm{diag}\left(X_{n,n-1},\ldots,X_{1,0}\right)=\mathrm{diag}\left(\frac{1}{\left|\left|v_n^\prime\right|\right|},\ldots,\frac{1}{\left|\left|v_1^\prime\right|\right|}\right)=\mathrm{diag}\left(\frac{\sqrt{\varphi_{n-1}}}{\sqrt{\varphi_{n}}},\ldots,\frac{1}{\sqrt{\varphi_1}}\right).
	\end{equation*} 
	Thus, $y_P$ is also an upper triangular matrix with diagonal 
	\begin{equation*}
	\mathrm{diag}\left(\frac{\sqrt{\varphi_{n}}}{\sqrt{\varphi_{n-1}}},\ldots,\sqrt{\varphi_1}\right).
	\end{equation*} 
	Now, by \cref{last_row_befor_iwa_dec} we get that $x_P$ is an upper triangular matrix with diagonal
	\begin{equation}\label{prab_mat_iwa_arch}
	\mathrm{diag}\left(0_{(m-n+1)n},\frac{\sqrt{\varphi_{n}}}{\sqrt{\varphi_{n-1}}}e_{n-2},\ldots,\frac{\sqrt{\varphi_{4}}}{\sqrt{\varphi_{3}}}e_2,\frac{\sqrt{\varphi_{3}}}{\sqrt{\varphi_{2}}}e_1+\frac{\sqrt{\varphi_{2}}}{\sqrt{\varphi_{1}}}e_n,\sqrt{\varphi_{1}}e_n\right).
	\end{equation}
	Now we can use \cref{prab_mat_iwa_arch} to evaluate \cref{section_val_on_iwa_arch}.
	\begin{equation*}
	f_{\omega_{\pi_\nu},s}\left(\tilde{w} \underline{t}^\Delta \tilde{w}^{-1}\right)=c(x_K)\omega_{\pi_{\nu}}(\sqrt{\varphi_{1}})\left|\sqrt{\varphi_{1}}\right|^{-mn(s+\frac{1}{2})}.
	\end{equation*}
	Plugging this result to \cref{sec_val_arc_place_wo_det} and then to \cref{sec_val_arc_place} gives
	\begin{equation*}
	f_{\omega_{\pi_\nu},s}\left(\tilde{\varepsilon}   \underline{t}^\Delta\right)= \left|\det \underline{t}\right|^{ms+\frac{m}{2}}c(\tilde{w})c(x_K)\omega_{\pi_{\nu}}(\sqrt{\varphi_{1}})\left|\sqrt{\varphi_{1}}\right|^{-mn(s+\frac{1}{2})}.
	\end{equation*}
	Recall that $x_K$ is the compact part in the Iwasawa decomposition of $\tilde{w} \underline{t}^\Delta \tilde{w}^{-1}$. So, by denoting $c(\tilde{\omega},\underline{t}):=c(\underline{\theta})c(\tilde{w})c(x_K)$ the proof is done.
\end{proof}

\begin{cor}\label{local_int_det_prop_arch}
	Assume that $\Re(s)>>0$. In the notation of \Cref{local_int_split_prop,sec_val_on_non_arch_det}, $I\left(f_{\omega_{\pi_\nu},s},v_{\pi_{\nu}}\right)(I_m)$ is a finite sum of integrals of the form
	\begin{equation}\label{local_int_det_eq_arch}
	\begin{split}
		\alpha\left(\tau,\phi,\phi'\right)\int \limits 	_{Z_n(k_\nu)\backslash T^-}&P(q_\nu^{-s}, q_\nu^s)\mu(\underline{t})c(\tilde{\omega},\underline{t})\left|\det \underline{t}\right|^{ms+\frac{m}{2}}\omega_{\pi_{\nu}}\left(1+\sum_{i=1}^{n-1}t_j^2\right)\\
		&\cdot \left|1+\sum_{i=1}^{n-1}t_j^2\right|^{-mn(s+\frac{1}{2})}\left(\pi_{\nu}(\underline{t})v_{\pi_{\nu}} ,\phi\right)d\underline{t}\cdot  \phi^{\prime},
	\end{split}
	\end{equation}
	where $P(q_\nu^{-s}, q_\nu^s)$ is a holomorphic function, $\mu(\underline{t})=\mu(K_{n,\nu}\underline{t}K_{n,\nu})$, and $c$ is the inner product of the space of $\rho_{\omega_{\pi_{\nu}},s}$ given in \cref{ind_space_inner_prod}.
\end{cor}

\begin{proof}
	The proof is the same as the proof of \Cref{local_int_det_prop}, except that this time we apply \Cref{sec_val_on_arch_det} instead of \Cref{sec_val_on_non_arch_det}.
\end{proof}

Next, we make use of the asymptotic behavior of matrix coefficients due to \cite{casselman1980jacquet}.
Embed $T^-$ in $\mathbb{C}^\Delta$: $\underline{t}\mapsto \left(\alpha(\underline{t})\right)_{\alpha\in \Delta}$. For each $s\in\mathbb{C}^\Delta$ define functions which are single-valued on $T^-$, multivalued on the complement of coordinate hyperplanes in $\mathbb{C}^\Delta$:
\begin{equation*}
	a^s\log^c a=\prod_{\alpha\in\Delta}\alpha(a)^{s_\alpha}\log^{c_\alpha}\alpha(a),
\end{equation*}
where $c\in \mathbb{N}^\Delta \subseteq \mathbb{C}^\Delta$, the set of integer vectors in $\mathbb{C}^\Delta$.
\begin{prop*}[Casselman, 1978]
	There exist finite sets $S\subseteq \mathbb{C}^\Delta, \mathcal{M}\subseteq \mathbb{N}^\Delta$ such that for every $v\in V$, $\tilde{v}\in \tilde{V}$, there exist functions $h_{s,c}\ (s\in S, c\in \mathcal{M})$ holomorphic in $T^-$ with
	\begin{equation}\label{asym_exp_arch}
		\left<\pi(a)v,\tilde{v}\right>=\sum_{s\in S,c\in \mathcal{M}} h_{s,c}a^s\log^c a,
	\end{equation}
	for $a\in T^-$.
\end{prop*}

Similarly to the non-Archimedean case we can deduce in particular,
\begin{cor}\label{mc_bound_arch}
	There exist $\lambda,m_a\in \bbr_+$ such that
	\begin{equation*}
	\left|\left\langle \pi(a)v,\tilde{v}\right\rangle \right| \leq \lambda ||a|| ^{m_a}.
	\end{equation*}
\end{cor}

\begin{prop}
	The integral $I\left(f_{\omega_{\pi_\nu},s},v_{\pi_{\nu}}\right)(I_m)$ absolutely converges in  $\Re(s)>>0$.
\end{prop}
\begin{proof}
	By \Cref{local_int_det_prop_arch}, the integral $I\left(f_{\omega_{\pi_\nu},s},v_{\pi_{\nu}}\right)(I_m)$ is a finite sum of integrals of the form \cref{local_int_det_eq_arch}. By \Cref{jacobian_of_kak_int_arch} we have  $\mu(\underline{t})=\sum_{\alpha\in \Sigma _+}\left| \sinh \alpha(H)\right|^{\dim \mathfrak{g}_\alpha}$.
	In addition, $||\underline{t}||=|t_1|^{-1}$. Thus, by \Cref{mc_bound_arch}, each of the integrals above is bounded in absolute value by
	\begin{equation*}
	\begin{split}
		\int \limits 	_{|t_1|\leq |t_2|\leq \ldots |t_{n-1}|\leq 1}&\prod_{i=1}^{n-1}e^{c_i t_i}|t_1|^{1-n-m_a+ms+m/2}|t_2|^{3-n+ms+m/2}\\&\cdots |t_{n-1}|^{n-3+ms+m/2}d^\times t_1  \cdots d^\times t_{n-1}.
	\end{split}
	\end{equation*}
	By separating variables, the last integral can be brought to the form
	\begin{equation*}
	\int \limits 	_{|t_1|\leq 1}e^{c_1 t_1}|t_1|^{1-n-m_a+ms+m/2}d^\times t_1\prod_{j=2}^{n-1}	\int \limits 	_{|t_j|\leq 1}e^{c_j t_j}|t_j|^{j+1-n+ms+m/2}d^\times t_j.
	\end{equation*}
	Hence, it is absolutely convergent iff $\Re(1-n-m_a+ms+m/2)\geq 0$.
\end{proof}
Moreover, the integral $I\left(f_{\omega_{\pi_\nu},s},v_{\pi_{\nu}}\right)(I_m)$  has meromorphic continuation. In order to see that we first apply \cref{asym_exp_arch} to \Cref{local_int_det_prop_arch} and obtain:
\begin{cor}\label{local_int_det_prop_arch2}
	Assume that $\Re(s)>>0$. In the notation of \Cref{local_int_split_prop,sec_val_on_arch_det}, $I\left(f_{\omega_{\pi_\nu},s},v_{\pi_{\nu}}\right)(I_m)$ is a finite sum of integrals of the form
	\begin{equation}\label{local_int_det_eq_arch2}
	\begin{split}
	\alpha\left(\tau,\phi,\phi'\right)\int \limits 	_{Z_n(k_\nu)\backslash T^-}&P(q_\nu^{-s}, q_\nu^s)\mu(\underline{t})c(\tilde{\omega},\underline{t})\left|\det \underline{t}\right|^{ms+\frac{m}{2}}\omega_{\pi_{\nu}}\left(1+\sum_{i=1}^{n-1}t_j^2\right)\\&\cdot \left|1+\sum_{i=1}^{n-1}t_j^2\right|^{-mn(s+\frac{1}{2})}\sum_{s\in S,c\in \mathcal{M}} h_{s,c}\underline{t}^s\log^c \underline{t} d\underline{t}\cdot  \phi^{\prime},
	\end{split}
	\end{equation}
	where $P(q_\nu^{-s}, q_\nu^s)$ is a holomorphic function, $\mu(\underline{t})=\sum_{\alpha\in \Sigma _+}\left| \sinh \alpha(H)\right|^{\dim \mathfrak{g}_\alpha}$, and $c$ is the inner product of the space of $\rho_{\omega_{\pi_{\nu}},s}$ given in \cref{ind_space_inner_prod}.
\end{cor}

\begin{prop}
	The integral $I\left(f_{\omega_{\pi_\nu},s},v_{\pi_{\nu}}\right)(I_m)$ has meromorphic continuation for  $\Re(s)>>0$.
\end{prop}
\begin{proof}
	By \Cref{local_int_det_prop}, $I\left(f_{\omega_{\pi_\nu},s},v_{\pi_{\nu}}\right)(I_m)$ is a finite sum of integrals of the form
	\begin{equation*}
	\alpha\left(\tau,\phi,\phi'\right)\beta\int \limits 	_{Z_n(k_\nu)\backslash T_\Theta^-(\varepsilon)}P(q_\nu^{-s}, q_\nu^s)\mu(\underline{t})c(X(\tilde{\varepsilon},\underline{t}))\left|\det \underline{t}\right|^{ms+\frac{m}{2}}\sum_{s\in S,c\in \mathcal{M}} h_{s,c}\underline{t}^s\log^c \underline{t} d\underline{t}\cdot  \phi^{\prime}.
	\end{equation*}
	where, $P(q_\nu^{-s}, q_\nu^s)$ is a holomorphic function,  $c(X(\tilde{\varepsilon},\underline{t}))$ does not depend on $s$, $\mu(\underline{t})= \sum_{\alpha\in \Sigma _+}\left| \sinh \alpha(H)\right|^{\dim \mathfrak{g}_\alpha}$ (by \Cref{jacobian_of_kak_int_arch}).
	We have $||\underline{t}||=|t_1|^{-1}$. Therefore, each of the integrals above equals
	\begin{equation*}
	\begin{split}
		\int \limits 	_{|t_1|\leq |t_2|\leq \ldots |t_{n-1}|\leq 1}&\prod_{i=1}^{n-1}e^{c_i t_i}|t_1|^{1-n+ms+m/2}|t_2|^{3-n+ms+m/2}\cdots |t_{n-1}|^{n-3+ms+m/2}\\&\cdot \prod_{i=1}^{r_\Theta}\chi_i(t_i)\left|t_{i}\right|^{p_{i}}\log^{p'_{i}}\left|t_{i}\right|d^\times t_1 d^\times t_2 \cdots d^\times t_{n-1}.
	\end{split}
	\end{equation*}
	By separating variables we find that the last integral equals
	\begin{equation}\label{merom_cont_int_sep_arch}
	\prod_{j=1}^{n-1}	\int \limits 	_{|t_j|\leq 1}e^{c_j t_j}\chi_j(t_j)|t_j|^{\alpha_j+ms}\log^{p'_{j}}\left|t_{j}\right|d^\times t_j.
	\end{equation}
	Each of the classic integrals in \cref{merom_cont_int_sep_arch} has meromorphic continuation and so does $I\left(f_{\omega_{\pi_\nu},s},v_{\pi_{\nu}}\right)(I_m)$.
\end{proof}

\section*{Acknowledgment}

I am grateful to my advisor, David Soudry, for his guidance, patience, support, and for many insightful discussions. 
I would also like to thank David Ginzburg for suggesting the integral construction in eq. (1.1), and for his kind help and advice.
My sincere gratitude goes to Solomon Friedberg and Elad Zelingher for their kind encouragement and practical suggestions.
This work was supported by the Israel science foundation.
\bibliographystyle{alpha}
\bibliography{references}

\begin{thebibliography}{CFGK16}

\bibitem[Cas80]{casselman1980jacquet}
William Casselman.
\newblock Jacquet modules for real reductive groups.
\newblock In {\em Proceedings of the International Congress of Mathematicians
  (Helsinki, 1978)}, volume 557563. Acad. Sci. Fennica Helsinki, 1980.

\bibitem[Cas95]{casselman1995introduction}
William Casselman.
\newblock Introduction to the theory of admissible representations of p-adic
  reductive groups, unpublished notes distributed by paul sally.
\newblock 1995.

\bibitem[CFGK16]{cai2016doubling}
Y.~{Cai}, S.~{Friedberg}, D.~{Ginzburg}, and E.~{Kaplan}.
\newblock Doubling constructions for covering groups and tensor product
  ${L}$-functions.
\newblock {\em arXiv e-prints}, 2016.
\newblock \href{https://arxiv.org/abs/1601.08240}{arXiv:1601.08240}.

\bibitem[GJ72]{godement1972local}
Roger {Godement} and Herv\'e {Jacquet}.
\newblock {\em {Zeta functions of simple algebras}}, volume 260.
\newblock Springer-Verlag, Berlin, 1972.

\bibitem[GPSR97]{ginzburg1997functions}
D.~Ginzburg, I.~Piatetski-Shapiro, and S.~Rallis.
\newblock {$L$} functions for the orthogonal group.
\newblock {\em Memoirs AMS, no. 611, Vol. 128}, 128(611), 1997.

\bibitem[GRS11]{ginzburg2011descent}
D.~Ginzburg, S.~Rallis, and D.~Soudry.
\newblock {\em The descent map from automorphic representations of {${\rm
  GL}(n)$} to classical groups}.
\newblock World Scientific, 2011.

\bibitem[GS18]{ginzburg2018two}
D.~Ginzburg and D.~Soudry.
\newblock Two identities relating {Eisenstein} series on classical groups.
\newblock {\em arXiv preprint}, 2018.
\newblock \href{https://arxiv.org/abs/1808.01572}{arXiv:1808.01572}.

\bibitem[GS19]{ginzburgintegrals}
D.~Ginzburg and D.~Soudry.
\newblock Integrals derived from the doubling method.
\newblock {\em IMRN}, 2019.
\newblock
  \href{https://academic.oup.com/imrn/advance-article-pdf/doi/10.1093/imrn/rnz147/29026032/rnz147.pdf}{rnz147}.

\bibitem[Haz22]{hazan2022asym}
Zahi Hazan.
\newblock A note on the asymptotic expansion of matrix coefficients over
  $p$-adic fields.
\newblock {\em arXiv preprint}, 2022.
\newblock \href{https://arxiv.org/abs/2211.15822}{arXiv:2211.15822}.

\bibitem[Hel84]{helgason1984groups}
Sigurdur Helgason.
\newblock {\em Groups and geometric analysis: integral geometry, invariant
  differential operators, and spherical functions}, volume~83.
\newblock Academic Press Incorporated, 1984.

\bibitem[II10]{ichino2010periods}
Atsushi Ichino and Tamutsu Ikeda.
\newblock On the periods of automorphic forms on special orthogonal groups and
  the gross--prasad conjecture.
\newblock {\em Geometric and Functional Analysis}, 19(5):1378--1425, 2010.

\bibitem[Ike94]{ikeda1994}
T.~Ikeda.
\newblock On the theory of {J}acobi forms and {F}ourier-{J}acobi coefficients
  of {E}isenstein series.
\newblock {\em J. Math. Kyoto Univ., no. 3, Vol. 34}, 34(3):615--636, 1994.

\bibitem[Kap13]{kaplan2013local}
Eyal Kaplan.
\newblock {\em Rankin-Selberg Convolutions for
  {$\rm{SO}_{2\ell+1}\times\rm{GL}_n$}}.
\newblock Tel-Aviv University. Raymond and Beverly Sackler faculty of exact
  sciences. School of mathematical sciences, 2013.

\bibitem[Kna01]{knapp2001representation}
Anthony~W Knapp.
\newblock Representation theory of semisimple groups: an overview based on
  examples.
\newblock 2001.

\bibitem[Mac98]{macdonald1998symmetric}
Ian~Grant Macdonald.
\newblock {\em Symmetric functions and Hall polynomials}.
\newblock Oxford university press, 1998.

\bibitem[M{\oe}g97]{moeglin1997quelques}
C.~M{\oe}glin.
\newblock Quelques propri\'{e}t\'{e}s de base des s\'{e}ries th\'{e}ta.
\newblock {\em J. Lie Theory, no. 2, Vol. 7}, 7(2):231--238, 1997.

\bibitem[Mui08]{muic2008geometric}
Goran Mui{\'c}.
\newblock A geometric construction of intertwining operators for reductive
  p-adic groups.
\newblock {\em manuscripta mathematica}, 125(2):241--272, 2008.

\bibitem[Wal03]{waldspurger2003formule}
J-L Waldspurger.
\newblock La formule de plancherel pour les groupes p-adiques. d’apres
  harish-chandra.
\newblock {\em Journal of the Institute of Mathematics of Jussieu},
  2(2):235--333, 2003.

\end{thebibliography}

\Addresses
\end{document}